\documentclass[11pt]{amsart}
\usepackage{fullpage}
\usepackage{times}

\let\into=\hookrightarrow

\DeclareSymbolFont{bbold}{U}{bbold}{m}{n}
\DeclareSymbolFontAlphabet{\mathbb}{bbold}
\newcommand{\LAMBDA}{\mathbb{\Lambda}}
\newcommand{\DELTA}{\mathbb{\Delta}}
\newcommand{\cell}{\mathbb}
\DeclareMathSymbol{\bbmu}{\mathord}{bbold}{"16}

\usepackage{amsfonts}
\usepackage[OT2,T1]{fontenc}
\DeclareSymbolFont{cyrletters}{OT2}{wncyr}{m}{n}
\DeclareMathSymbol{\Sha}{\mathalpha}{cyrletters}{"58}
\usepackage[initials,lite]{amsrefs}

\usepackage{enumitem,calc}

\usepackage[cmtip,curve,all]{xy}

\renewcommand{\*}{}

\let\To=\to
\let\Box=\square
\input diagxy
\let\too=\to
\let\SQUARE=\square
\let\square=\Box
\renewcommand{\to}{\To}

\numberwithin{equation}{section}

\theoremstyle{definition}

\newtheorem{theorem}{Theorem}[section]
\newtheorem*{theorem*}{Theorem}
\newtheorem{proposition}[theorem]{Proposition}
\newtheorem{corollary}[theorem]{Corollary}
\newtheorem{lemma}[theorem]{Lemma}
\newtheorem*{lemma*}{Lemma}
\newtheorem{definition}[theorem]{Definition}
\newtheorem*{definition*}{Definition}
\newtheorem{example}[theorem]{Example}

\DeclareMathOperator{\HOM}{\underline{Hom}}
\DeclareMathOperator{\End}{End}
\DeclareMathOperator{\Map}{Hom}

\DeclareMathOperator{\cosk}{cosk}
\DeclareMathOperator{\sk}{sk}

\newcommand{\p}{\partial}
\newcommand{\bull}{\bullet}
\newcommand{\R}{\mathbb{R}}

\newcommand{\N}{\mathbb{N}}
\newcommand{\K}{\mathbb{K}}

\newcommand{\CF}{\mathcal{F}}
\newcommand{\CV}{\mathcal{V}}
\newcommand{\CW}{\mathcal{W}}

\renewcommand{\o}{\otimes}
\DeclareMathOperator{\MC}{\mathsf{MC}}
\DeclareMathOperator{\ad}{ad}

\newcommand{\om}{\omega}
\newcommand{\V}{\mathcal{V}}
\newcommand{\W}{\mathcal{W}}

\def\hookrightarrow{\mathchoice
  {\DOTSB\lhook\joinrel\relbar\joinrel\rightarrow}
  {\DOTSB\lhook\joinrel\rightarrow}
  {\DOTSB\lhook\joinrel\rightarrow}
  {\DOTSB\lhook\joinrel\rightarrow}}

\newcommand{\Cinf}{$C^\infty$}

\newcommand{\TT}{\mathbb{T}}
\newcommand{\T}{\mathsf{T}}
\newcommand{\ts}{\tilde{s}}

\renewcommand{\[}{[\![}
\renewcommand{\]}{]\!]}

\newcommand{\X}{\mathcal{X}}

\DeclareMathOperator{\GL}{GL}

\newcommand{\VA}{\mathit{VA}}
\newcommand{\UA}{\mathit{UA}}

\newcommand{\D}{\Delta}
\renewcommand{\L}{\Lambda}
\newcommand{\GG}{\mathbb{G}}
\newcommand{\G}{\mathsf{G}}
\renewcommand{\P}{\mathbb{P}}

\title{Geometric higher groupoids and categories}

\author{Kai Behrend}

\address{Department of Mathematics, University of British Columbia}

\email{behrend@math.ubc.ca}

\author{Ezra Getzler}

\address{Department of Mathematics, Northwestern University}

\email{getzler@northwestern.edu}

\thanks{We are grateful to Nick Roszenblyum for remarking that the
  path space of a quasi-category $X$ is $\Map(\DELTA^1,X)$, and to
  Jesse Wolfson for many helpful discussions. The first author thanks
  Imperial College for its hospitality during the period when this
  paper was begun. The second author thanks the University of Geneva
  for its hospitality during the period when it was completed, and the
  Simons Foundation for support under a Collaboration Grant for
  Mathematicians.}

\begin{document}

\begin{abstract}
  In an enriched setting, we show that higher groupoids and higher
  categories form categories of fibrant objects, and that the nerve of
  a differential graded algebra is a higher category in the category
  of algebraic varieties.
\end{abstract}

\maketitle  

This paper develops a general theory of higher groupoids in a category
$\CV$. We consider a small category $\CV$ of \textbf{spaces}, together
with a subcategory of \textbf{covers}, satisfying the following
axioms:
\begin{enumerate}[label=(D\arabic*),labelwidth=\widthof{(D3)},leftmargin=!]
\item \label{finite} $\CV$ has finite limits;
\item the pullback of a cover is a cover;
\item \label{cancellation} if $f$ is a cover and $g\*f$ is a cover,
  then $g$ is a cover.
\end{enumerate}

These axioms are reminiscent of those for a category of smooth
morphisms $\mathbf{P}$ of Toen and Vezzosi (\cite{tv},
Assumption~1.3.2.11). A topos satisfies these axioms, with
epimorphisms as covers; so do the category of schemes, with surjective
\'etale morphisms, smooth epimorphisms, or faithfully flat morphisms
as covers, and the category of Banach analytic spaces, with surjective
submersions as covers. We call a category satisfying these axioms a
\textbf{descent category}. We call a simplicial object in a descent
category a \textbf{simplicial space}.

A finite simplicial set is a simplicial set with a finite number of
\emph{degenerate} simplices. Given a simplicial space $X$ and a
finite simplicial set $T$, let
\begin{equation*}
  \Map(T,X)
\end{equation*}
be the space of simplicial morphisms from $T$ to $X$; it is a finite
limit in $\CV$, and its existence is guaranteed by \ref{finite}.

Let $\L^n_i\subset\D^n$ be the \textbf{horn}, consisting of the union
of all but the $i$th face of the $n$-simplex:
\begin{equation*}
  \L^n_i = \bigcup_{j\ne i} \p_j\D^n .
\end{equation*}
A simplicial set $X$ is the nerve of a groupoid precisely when the
induced morphism
\begin{equation*}
  X_n\to\Map(\L^n_i,X)  
\end{equation*}
is an isomorphism for $n>1$.

On the other hand, given a simplicial abelian group $A$, the
associated complex of normalized chains vanishes above degree $k$ if
and only if the morphism $A_n\to\Map(\L^n_i,A)$ is an isomorphism for
$n>k$. Motivated by these examples, Duskin defined a $k$-groupoid to
be a simplicial set $X$ such that the morphism $X_n\to\Map(\L^n_i,X)$
is surjective for $n>0$ and bijective for $n>k$. (See
Duskin~\cite{Duskin} and Glenn~\cite{Glenn}. In their work,
$k$-groupoids are called ``$k$-dimensional hypergroupoids.'')

In this paper, we generalize Duskin's theory of $k$-groupoids to
descent categories: Pridham takes a similar approach in
\cite{Pridham}.
\begin{definition*}
  Let $k$ be a natural number. A simplicial space $X$ in a descent
  category $\CV$ is a \textbf{\boldmath{$k$}-groupoid} if, for each
  $0\le i\le n$, the morphism
  \begin{equation*}
    X_n \too \Map(\L^n_i,X)
  \end{equation*}
  is a cover for $n>0$, and an isomorphism for $n>k$.
\end{definition*}

Denote by $s_k\CV$ the category of $k$-groupoids, with morphisms the
simplicial morphisms of the underlying simplicial spaces. Thus, the
category $s_0\CV$ of $0$-groupoids is equivalent to $\CV$, while the
category $s_1\CV$ of $1$-groupoids is equivalent to the category of
Lie groupoids in $\CV$, that is, groupoids such that the source and
target maps are covers. (The equivalence is induced by mapping a Lie
groupoid to its nerve.)

\begin{definition*}
  A morphism $f:X\to Y$ between $k$-groupoids is a \textbf{fibration}
  if, for each $n>0$ and $0\le i\le n$, the morphism
  \begin{equation*}
    X_n \too \Map(\L^n_i,X)\times_{\Map(\L^n_i,Y)}Y_n
  \end{equation*}
  is a cover. It is a \textbf{hypercover} if, for each $n\ge0$, the
  morphism
  \begin{equation*}
    X_n \too \Map(\p\D^n,X)\times_{\Map(\p\D^n,Y)}Y_n
  \end{equation*}
  is a cover. It is a \textbf{weak equivalence} if there is a
  $k$-groupoid $P$ and hypercovers $p:P\to X$ and $q:P\to Y$ such that
  $f=qs$, where $s$ is a section of $p$.
\end{definition*}

Every $k$-groupoid is \textbf{fibrant}: that is, the unique morphism
with target the terminal object $e$ is a fibration. Every hypercover
is a fibration.

The following is the first main result of this paper: for the
definition of a category of fibrant objects, see Definition~\ref{CFO}.
\begin{theorem*}
  The category of $k$-groupoids $s_k\CV$ is a category of fibrant
  objects.
\end{theorem*}

We will prove the following more direct characterization of weak
equivalences in Section~5.
\begin{theorem*}
  A morphism $f:X\to Y$ between $k$-groupoids is a \textbf{weak
    equivalence} if and only if, for each $n\ge0$, the morphisms
  \begin{equation*}
    X_n\times_{Y_n}Y_{n+1} \too
    \Map(\p\D^n,X)\times_{\Map(\p\D^n,Y)}\Map(\L^{n+1}_{n+1},Y)
  \end{equation*}
  are covers.
\end{theorem*}

Parallel to the theory of $k$-groupoids, there is a theory of
$k$-categories, modeled on the theory of complete Segal spaces (Rezk
\cite{Rezk}). In the case where $\CV$ is the category of sets, these
are truncated weak Kan complexes in the sense of Boardman and Vogt
\cite{BV}. Weak Kan complexes were studied further by Joyal
\cite{Joyal}, who calls them quasi-categories, and by Lurie
\cite{Lurie}, who calls them $\infty$-categories.

The \textbf{thick \boldmath{$n$}-simplex} is the simplicial set
$\DELTA^n = \cosk_0\D^n$. Just as $\D^n$ is the nerve of the category
with objects $\{0,\dots,n\}$ and a single morphism from $i$ to $j$ if
$i\le j$, $\DELTA^n$ is the nerve of the groupoid $\[n\]$
with objects $\{0,\dots,n\}$ and a single morphism from $i$ to $j$ for
all $i$ and $j$. In other words, just as the $k$-simplices of the
$n$-simplex are monotone functions from $\{0,\dots,k\}$ to
$\{0,\dots,n\}$, the $k$-simplices of the thick simplex are \emph{all}
functions from $\{0,\dots,k\}$ to $\{0,\dots,n\}$. What we call the
thick simplex goes under several names in the literature:
Rezk~\cite{Rezk} denotes it $E(n)$, while Joyal and Tierney~\cite{JT}
use the notation $\D'[n]$.

\begin{definition*}
  Let $k$ be a positive integer. A simplicial space $X$ in a descent
  category $\CV$ is a \textbf{\boldmath{$k$}-category} if for each
  $0<i<n$, the morphism
  \begin{equation*}
    X_n \to \Map(\L^n_i,X) ,
  \end{equation*}
  is a cover for $n>1$ and an isomorphism for $n>k$, and the morphism
  \begin{equation*}
    \Map(\DELTA^1,X) \to X_0
  \end{equation*}
  induced by the inclusion of a vertex $\D^0\hookrightarrow\DELTA^1$
  is a cover.
\end{definition*}
  
In a topos, where all epimorphisms are covers, the last condition is
automatic, since these morphisms have the section
$X_0\to\Map(\DELTA^1,X)$ induced by the projection from $\DELTA^1$ to
$\D^0$.

Associated to a $k$-category $X$ is the simplicial space $\GG(X)$,
defined by
\begin{equation*}
  \GG(X)_n = \Map(\DELTA^n,X) .
\end{equation*}
The formation of $\GG(X)_n$, while appearing to involve an infinite
limit, is actually isomorphic to a finite limit, since (see
Lemma~\ref{coskeletal})
\begin{equation*}
  \Map(\DELTA^n,X) \cong \cosk_{k+1}X_n = \Map(\sk_{k+1}\DELTA^n,X) ,
\end{equation*}
and $\sk_{k+1}\DELTA^n$, the $(k+1)$-skeleton of $\DELTA^n$, is a
finite simplicial complex.

The following theorem is useful in constructing examples of
$k$-groupoids.
\begin{theorem*}
  If $X$ is a $k$-category, $\GG(X)$ is a $k$-groupoid.
\end{theorem*}

In fact, $k$-categories also form a category of fibrant objects.

\begin{definition*}
  A morphism $f:X\to Y$ of $k$-categories is a
  \textbf{quasi-fibration} if for $0<i<n$, the morphism
  \begin{equation*}
    X_n \to \Map(\L^n_i,X)\times_{\Map(\L^n_i,Y)}Y_n
  \end{equation*}
  is a cover, and the morphism
  \begin{equation*}
    \Map(\DELTA^1,X) \to X_0\times_{Y_0}\Map(\DELTA^1,Y)
  \end{equation*}
  induced by the inclusion of a vertex $\D^0\hookrightarrow\DELTA^1$
  is a cover. It is a \textbf{hypercover} if, for each $n\ge0$, the
  morphism
  \begin{equation*}
    X_n \too \Map(\p\D^n,X)\times_{\Map(\p\D^n,Y)}Y_n
  \end{equation*}
  is a cover. (This is the same definition as for $k$-groupoids,
  except that now $X$ and $Y$ are $k$-categories.) It is a
  \textbf{weak equivalence} if there is a $k$-category $P$ and
  hypercovers $p:P\to X$ and $q:P\to Y$ such that $f=qs$, where $s$ is
  a section of $p$.
\end{definition*}

\begin{theorem*}
  \mbox{}
  \begin{enumerate}[label=\roman*)]
  \item The category of $k$-categories is a category of fibrant
    objects.
  \item The functor $\GG$ is an exact functor: it takes
    quasi-fibrations to fibrations, pullbacks of quasi-fibrations to
    pullbacks, and hypercovers to hypercovers.
  \end{enumerate}
\end{theorem*}

We also have the following more direct characterization of weak
equivalences between $k$-categories, proved in Section~6. Recall that
if $S$ and $T$ are simplicial sets, then their \textbf{join}
$K\star L$ is the simplicial set
\begin{equation*}
  (S\star T)_k = S_k \sqcup T_k \sqcup \bigsqcup_{j=0}^{k-1} S_j
  \times T_{k-j-1}
\end{equation*}
\begin{theorem*}
  A morphism $f:X\to Y$ of $k$-categories is a weak equivalence if and
  only if the morphism
  \begin{equation*}
    X_0 \times_{Y_0} \Map(\DELTA^1,Y) \too Y_0
  \end{equation*}
  is a cover, and the morphisms
  \begin{multline*}
    X_n \times_{Y_n} \Map(\DELTA^1\star\D^{n-1},Y) \\
    \too \Map(\p\D^n,X) \times_{\Map(\p\D^n,Y)}
    \Map(\DELTA^1\star\p\D^{n-1}\cup\LAMBDA^1_0\star\D^{n-1},Y)
  \end{multline*}
  are covers for $n>0$.
\end{theorem*}

In a finite-dimensional algebra or a Banach algebra, invertibility is
an open condition. To formulate this property in our general setting,
we need the notion of a regular descent category.

A morphism in a category is an effective epimorphism if it equals its
own coimage. (We recall the definition of the coimage of a morphism in
Section~2.)
\begin{definition*}
  A \textbf{subcanonical} descent category is a descent category such
  that every cover is an effective epimorphism.
\end{definition*}

\begin{definition*}
  A \textbf{regular} descent category is a subcanonical descent
  category with a subcategory of \textbf{regular} morphisms,
  satisfying the following axioms:
  \begin{enumerate}[label=(R\arabic*),labelwidth=\widthof{(R3)},leftmargin=!]
  \item every cover is regular;
  \item the pullback of a regular morphism is regular;
  \item every regular morphism has a coimage, and its coimage is a
    cover.
  \end{enumerate}
\end{definition*}

All of the descent categories that we consider are regular. In the
case of a topos, we take all of the morphisms to be regular. When
$\CV$ is the category of schemes with covers the surjective \'etale
(respectively smooth or flat) morphisms, the regular morphisms are the
the \'etale (respectively smooth or flat) morphisms. When $\CV$ is the
category of Banach analytic spaces with covers the surjective
submersions, the regular morphisms are the submersions.
\begin{definition*}
  A $k$-category in a regular descent category $\V$ is
  \textbf{regular} if the morphism
  \begin{equation*}
    \Map(\DELTA^1,X) \to \Map(\D^1,X) = X_1
  \end{equation*}
  is regular.
\end{definition*}

\begin{theorem*}
  Let $\V$ be a regular descent category, and let $X$ be a regular
  $k$-category in $\V$. Then for all $n\ge0$, the morphism
  \begin{equation*}
    \Map(\DELTA^n,X) \to \Map(\D^n,X) = X_n
  \end{equation*}
  is regular. Let $\G(X)_n$ be the image of this morphism (that is,
  the codomain of its coimage). Then the spaces $\G(X)$ form a
  simplicial space, this simplicial space is a $k$-groupoid, and the
  induced morphism
  \begin{equation*}
    \GG(X) \to \G(X)
  \end{equation*}
  is a hypercover.
\end{theorem*}

In fact, as shown by Joyal (Corollary~1.5, \cite{Joyal}), $\G(X)_n$ is
the space of $n$-simplices of $X$ such that for each inclusion
$\D^1\into\D^n$, the induced $1$-simplex lies in $\G(X)_1$. The
simplices of $\G(X)_1$ are called \textbf{quasi-invertible}.

In the case where $\V$ is the category of sets, this theorem relates
two different $k$-groupoids associated to a $k$-category: the
$k$-groupoid $\GG(X)$ was introduced by Rezk \cite{Rezk} and further
studied by Joyal and Tierney \cite{JT}, while the $k$-groupoid $\G(X)$
was introduced by Joyal \cite{Joyal}.

In the last section of this paper, we construct examples of
$k$-groupoids associated to differential graded algebras over a
field. Let $A$ be a differential graded algebra such that $A^i$ is
finite-dimensional for all $i$. The \textbf{Maurer-Cartan locus}
$\MC(A)$ of $A$ is the affine variety
\begin{equation*}
  \MC(A) = Z(da+a^2) \subset A^1 .
\end{equation*}
If $K$ is a finite simplicial set, let $C^\bull(K)$ be the
differential graded algebra of normalized simplicial cochains on
$K$. The \textbf{nerve} of $A$ is the simplicial scheme
\begin{equation*}
  N_nA = \MC(C^\bull(\Delta)\o A) .
\end{equation*}
This simplicial scheme has also been discussed by Lurie \cite{Lurie}.
\begin{theorem*}
  Let $A$ be a differential graded algebra finite-dimensional in each
  degree and vanishing in degree $-k$ and below. The nerve $N_\bull A$
  of $A$ is a regular $k$-category in the descent category of schemes
  (with surjective submersions as covers).
\end{theorem*}

The $k$-groupoid $\N_\bull A=\GG(NA)$ is the simplicial scheme
\begin{equation*}
  \N_nA = \MC(C^\bull(\DELTA^n)\o A) .
\end{equation*}
We see that $\N_\bull A$ and $\G(N_\bull A)$ are $k$-groupoids, and
that the simplicial morphism
\begin{equation*}
  \N_\bull A \to \G(N_\bull A)
\end{equation*}
is a hypercover. The statement that $\G(N_\bull A)$ is a $k$-groupoid
has also been proved by Benzeghli~\cite{Benzeghli}.

This theorem has an evident generalization to differential graded
categories. It may also be generalized to differential graded Banach
algebras, in which case the nerve is a $k$-category in the descent
category of Banach analytic spaces. There is also a more refined
version of the theorem in which the Maurer-Cartan locus is taken in
the category of derived schemes; this will be the topic of a sequel to
this paper.

\section{Categories of fibrant objects}

\begin{definition}
  \label{CFO}
  A \textbf{category with weak equivalences} is a category $\CV$
  together with a subcategory $\CW\subset\CV$ containing all
  isomorphisms, such that whenever $f$ and $g$ are composable
  morphisms such that $gf$ is a weak equivalence, then $f$ is a weak
  equivalence if and only if $g$ is.
\end{definition}

Associated to a small category with weak equivalences is its
simplicial localization $L(\CV,\CW)$. This is a category enriched in
simplicial sets, with the same objects as $\CV$, which refines the
usual localization. (In fact, the morphisms of the localization are
the components of the simplicial sets of morphisms of $L(\CV,\CW)$.)
The simplicial localization was introduced by Dwyer and Kan
\cites{dk1,dk2}, and studied further in Dwyer and Kan \cite{dk3},
Weiss \cite{weiss}, and Cisinski \cite{cisinski}: one may even say
that abstract homotopy theory is the study of simplicial
localizations. The simplicial category of \textbf{\boldmath$k$-stacks}
is the simplicial localization of the category of $k$-groupoids.

Categories of fibrant objects, introduced by Brown \cite{Brown}, form
a class of categories with weak equivalences for which the simplicial
localization is quite tractable: the simplicial sets of morphisms
between objects may be realized as nerves of certain categories of
diagrams.
\begin{definition}
  A \textbf{category of fibrant objects} $\CV$ is a small category
  with weak equivalences $\CW$ together with a subcategory
  $\CF\subset\CV$ of fibrations, satisfying the following
  axioms. Here, we refer to morphisms which are both fibrations and
  weak equivalences as \textbf{trivial fibrations}.
  \begin{enumerate}[label=(F\arabic*),labelwidth=\widthof{(F5)},leftmargin=!]
  \item \label{F:fibrant} There exists a terminal object $e$ in $\CV$,
    and any morphism with target $e$ is a fibration.
  \item \label{F:pullback} Pullbacks of fibrations are fibrations.
  \item \label{F:trivial-pullback} Pullbacks of trivial fibrations are
    trivial fibrations.
  \item \label{factorize} Every morphism $f:X\to Y$ has a
    factorization
    \begin{equation*}
      \begin{xy}
        \Atriangle/{<-}`>`>/<500,200>[P`X`Y;r`q`f]
      \end{xy}
    \end{equation*}
    where $r$ is a weak equivalence and $q$ is a fibration.
  \end{enumerate}
\end{definition}

An object $X$ such that the morphism $X\to e$ is a fibration is called
fibrant: Axiom~\ref{F:fibrant} states that every object is fibrant.

The reason for the importance of categories of fibrant objects is that
they allow a simple realization of the simplicial localization
$L(\CV,\CW)$ solely in terms of the trivial fibrations. Namely, by a
theorem of Cisinski \cite{cisinski}*{Proposition~3.23}, the simplicial
Hom-set $\HOM(X,Y)$ of morphisms from $X$ to $Y$ in the simplicial
localization of a category of fibrant objects is the nerve of the
category whose objects are the spans
\begin{equation*}
  \begin{xy}
    \Atriangle/{->}`{->}`{}/<400,300>[P`X`Y;f`g`]
  \end{xy}
\end{equation*}
where $f$ is a trivial fibration, and whose morphisms are commuting
diagrams
\begin{equation*}
  \begin{xy}
    \Atriangle/{->}`{->}`{}/<400,400>[P_0`X`Y;f_0`g_0`]
    \morphism(400,400)<0,-800>[P_0`P_1;h]
    \Vtriangle(0,-400)/{}`{<-}`{<-}/<400,400>[X`Y`P_1;`f_1`g_1]
  \end{xy}
\end{equation*}
(In the examples considered in this paper, in which the factorizations
in the category of fibrant objects are functorial, this result already
follows from the papers \cites{dk1,dk2}.)

The following lemma is due to Brown; the idea behind the proof goes
back to Serre's thesis (Chap\^{\i}tre~IV, Proposition 4 \cite{Serre}).
\begin{lemma}
  \label{P(f)}
  The weak equivalences of a category of fibrant objects are
  determined by the trivial fibrations: a morphism $f$ is a weak
  equivalence if and only if it factorizes as a composition $q\*s$,
  where $q$ is a trivial fibration and $s$ is a section of a trivial
  fibration.
\end{lemma}
\begin{proof}
  Let $Y$ be an object of $\CV$. The diagonal $Y\to Y\times Y$ has a
  factorization into a weak equivalence followed by a fibration:
  \begin{equation*}
    \begin{xy}
      \morphism<600,0>[Y`PY;s] \morphism(600,0)<700,0>[PY`Y \times Y
      \, .;\p_0\times\p_1]
    \end{xy}
  \end{equation*}
  The object $PY$ is called a \textbf{path space} of $Y$.

  Since $Y$ is fibrant, the two projections from $Y\times Y$ to $Y$
  are fibrations, since they are pullbacks of the fibration $Y\to e$:
  it follows that the morphisms
  \begin{equation*}
    \p_0,\p_1:PY\too Y
  \end{equation*}
  are fibrations as well. Since they are weak equivalences (by
  saturation of weak equivalences), they are actually trivial
  fibrations.

  Given a morphism $f:X\to Y$, form the pullback
  \begin{equation*}
    \begin{xy}
      \Square[P(f)`PY`X`Y;\pi`p(f)`\p_0`f]
    \end{xy}
  \end{equation*}
  We see that the projection $p(f):P(f)\to X$ is a trivial fibration,
  with section $s(f):X\to P(f)$ induced by the morphisms $s:Y\to PY$
  and $f:X\to Y$.

  We may also express $P(f)$ as a pullback
  \begin{equation*}
    \begin{xy}
      \Square[P(f)`PY`X\times Y`Y\times Y;%
      \pi`p(f)\times q(f)`\p_0\times\p_1`f\times1_Y]
    \end{xy}
  \end{equation*}
  This shows that $p(f)\times q(f)$ is a fibration. Composing with the
  projection $X\times Y\to Y$, which is a fibration since $X$ is
  fibrant, it follows that $q(f):P(f)\to Y$ is a fibration. In this
  way, we obtain the desired factorization of $f$:
  \begin{equation*}
    \begin{xy}
      \Atriangle/{<-}`{->}`{->}/<800,400>[P(f)`X`Y;s(f)`q(f)`f]
    \end{xy}
    \qedhere
  \end{equation*}
\end{proof}

The proof of this lemma shows that Axiom \ref{factorize} is implied by
the following special case:
\begin{enumerate}[label=(F\arabic*$\ast$),start=4,labelwidth=\widthof{(F4*)},leftmargin=!]
\item Each diagonal morphism $f:X\to X\times X$ has a factorization
  \begin{equation*}
    \begin{xy}
      \Atriangle/{<-}`>`>/<500,200>[P`X`X\times X;r`q`f]
    \end{xy}
  \end{equation*}
  where $r$ is a weak equivalence and $q$ is a fibration.
\end{enumerate}

\section{Descent categories}

Recall the axioms for a descent category, which we stated in the
introduction.
\begin{enumerate}[label=(D\arabic*),labelwidth=\widthof{(D3)},leftmargin=!]
\item $\CV$ has finite limits;
\item the pullback of a cover is a cover;
\item if $f$ is a cover and $g\*f$ is a cover, then $g$ is a cover.
\end{enumerate}

The covers in a descent category form a pre-topology on $\CV$
(Grothendieck and Verdier~\cite{Verdier}) with the special property
that every cover consists of a single morphism: this class of
pre-topologies will be sufficient for our purposes. Axiom
\ref{cancellation}, which has no counterpart in the usual theory of
Grothendieck topologies, plays a key role in this article.

The above axioms hold in the category of Kan complexes, with the
trivial fibrations as covers.  In the study of higher stacks, an
additional axiom is sometimes assumed, that covers are closed under
formation of retracts (c.f.\ Henriques \cite{Henriques}); we will not
need this axiom here.

\begin{example} \mbox{}
  \begin{enumerate}[label=\alph*)]
  \item The category of schemes is a descent category, with surjective
    \'etale, smooth or flat morphisms as the covers.
  \item The category of analytic spaces is a descent category, with
    surjective submersions as covers. A morphism $f:X\to Y$ of
    analytic spaces is a submersion if for every point $x\in X$, there
    is a neighbourhood $U$ of $x$, a neighbourhood $V$ of $f(x)$, and
    an isomorphism of analytic spaces $U\cong B\times V$ for which $f$
    is identified with projection to $V$, where $B$ is an open ball in
    a complex vector pace.
  \item More generally, by Douady \cite{douady}, the category of
    Banach analytic spaces is a descent category, again with
    surjective submersions as covers.
  \end{enumerate}
\end{example}

\begin{example}
  A \Cinf-ring (Dubuc \cite{dubuc}) is a real vector space $R$ with
  operations
  \begin{equation*}
    \rho_n : A(n) \times R^n \to R , \quad n\ge0 ,
  \end{equation*}
  where $A(n)=C^\infty(\R^n,\R)$. For every natural number $n$ and
  $n$-tuple $(m_1,\dots,m_n)$, the following diagram must commute:
  \begin{equation*}
    \begin{xy}
      \SQUARE<2200,500>[A(n) \times A(m_1) \times \cdots \times A(m_n)
      \times R^{m_1}\times\cdots\times R^{m_n}`A(n) \times R^n`
      A(m_1+\dots+m_n) \times R^{m_1}\times\cdots\times R^{m_n}`R;
      A(n)\times\rho_{m_1}\times\cdots\times\rho_{m_n}``\rho_n`\rho_{m_1+\cdots+m_n}]
    \end{xy}
  \end{equation*}
  The category of \Cinf-schemes is the opposite of the category of
  \Cinf-rings. This category has finite limits, and contains the
  category of differentiable manifolds as a full subcategory. It is
  also a descent category, with covers the surjective submersions. The
  category of Lie groupoids in the category of \Cinf-schemes is a
  natural generalization of the category of Lie groupoids in the usual
  sense: one of the results of this paper is that it is a category of
  fibrant objects.
\end{example}

The \textbf{kernel pair} of a morphism $f:X\to Y$ in a category with
finite limits is the diagram
\begin{equation*}
  \begin{xy}
    \morphism/{@{>}@<3pt>}/<600,0>[X\times_YX`X;]
    \morphism/{@{>}@<-3pt>}/<600,0>[X\times_YX`X;]
  \end{xy}
\end{equation*}
The coequalizer $p$ of the kernel pair of $f$, if it exists, is called
the \textbf{coimage} of $f$:
\begin{equation*}
  \begin{xy}
    \morphism/{@{>}@<3pt>}/<600,0>[X\times_YX`X;]
    \morphism/{@{>}@<-3pt>}/<600,0>[X\times_YX`X;]
    \morphism(600,0)<400,0>[X`Z;p]
    \morphism(1000,0)/{.>}/<400,0>[Z`Y;i]
    \morphism(600,0)/{@{>}@/^20pt/}/<800,0>[X`Y;f]
  \end{xy}
\end{equation*}
The \text{image} of $f$ is the morphism $i:Z\to Y$.

A morphism $f:X\to Y$ in a category $\CV$ is an \textbf{effective
  epimorphism} if $p$ equals $f$, in the sense that $i$ is an
isomorphism. One of the reasons for the importance of effective
epimorphisms is that pullback along an effective epimorphism is
conservative (reflects isomorphisms).
\begin{definition}
  A descent category is \textbf{subcanonical} if covers are effective
  epimorphisms.
\end{definition}
All of the descent categories which we have defined above have this
property.

In the study of categories, regular categories play a special role:
these are categories with finite limits in which pullbacks of
effective epimorphisms are effective epimorphisms, and kernel pairs
have coequalizers. Such categories share some basic properties with
the category of sets: in particular, every morphism factors into an
effective epimorphism followed by a monomorphism, and such a
factorization is unique up to isomorphism.

Recall from the introduction that a regular descent category is a
subcanonical descent category $\CV$ together with a subcategory of
\textbf{regular} morphisms satisfying the following axioms.
\begin{enumerate}[label=(R\arabic*),labelwidth=\widthof{(R3)},leftmargin=!]
\item every cover is regular;
\item the pullback of a regular morphism is regular;
\item every regular morphisms has a coimage, and its coimage is a
  cover.
\end{enumerate}
The following lemma is an example of the way in which a number of
properties of regular categories, suitably reformulated, extend to
regular descent categories.
\begin{lemma}
  Let $\CV$ be a regular descent category, and consider the
  factorization of a regular morphism $f:X\to Y$ into a cover
  $p:X\to Z$ followed by a morphism $i:Z\to Y$. Then $i$ is a
  monomorphism.
\end{lemma}
\begin{proof}
  The morphism
  \begin{equation*}
    p\times_Yp : X\times_YX \to Z\times_YZ
  \end{equation*}
  is the composition of a pair of covers
  \begin{equation*}
    X\times_YX \too^{X\times_Yp} X\times_YZ \too^{p\times_YZ}
    Z\times_YZ ,
  \end{equation*}
  hence itself a cover. The two compositions $\pi_1\circ(p\times_Yp)$
  and $\pi_2\circ(p\times_Yp)$ from $X\times_YX$ to $Z$ are
  equal. Since $p\times_Yp$ is a cover, it is an effective
  epimorphism, hence $\pi_1=\pi_2:Z\times_YZ\to Z$. This implies that
  $i:Z\to Y$ is a monomorphism.
\end{proof}

\section{$k$-groupoids}

Fix a descent category $\CV$.  We refer to simplicial objects taking
values in $\CV$ as \textbf{simplicial spaces}. Denote the category of
simplicial spaces by $s\CV$.
\begin{definition}
  Let $T$ be a finite simplicial set, and let $S\into T$ be a
  simplicial subset. If $f:X\to Y$ is a morphism of simplicial spaces,
  define the space
  \begin{equation*}
    \Map(S\into T,f) = \Map(S,X)\times_{\Map(S,Y)}\Map(T,Y) .
  \end{equation*}
This space parametrizes simplicial maps from $T$ to $Y$
with a lift to $X$ along $S$.
\end{definition}

Let $n\ge0$ be a natural number. The \textbf{matching space}
$\Map(\p\D^n,X)$ of a simplicial space $X$ (also denoted $M_n(X)$) is
the finite limit $\Map(\p\D^n,X)$, which represents simplicial
morphisms from the boundary $\p\D^n$ of the $n$-simplex $\D^n$ to
$X$. More generally, the matching space of a simplicial morphism
$f:X\to Y$ between simplicial spaces is the finite limit
\begin{equation*}
  \Map(\p\D^n\into\D^n,f) =
  \Map(\p\D^n,X)\times_{\Map(\p\D^n,Y)} Y_n .
\end{equation*}

\begin{definition}
  A simplicial morphism $f:X\to Y$ in $s\CV$ is a \textbf{hypercover}
  if for all $n\ge0$ the morphism
  \begin{equation*}
    X_n \too \Map(\p\D^n\into\D^n,f)
  \end{equation*}
  is a cover.
\end{definition}

\begin{lemma}
  \label{cofibration}
  Let $T$ be a finite simplicial set, and let $S\into T$ be a
  simplicial subset. If $f:X\to Y$ is a hypercover, then the induced
  morphism
  \begin{equation*}
    \Map(T,X) \too \Map(S\into T,f)
  \end{equation*}
  is a cover.
\end{lemma}
\begin{proof}
  There is a finite filtration of $T$
  \begin{equation*}
    S=F_{-1}T \subset F_0T \subset F_1T \subset \dots \subset T
  \end{equation*}
  satisfying the following conditions:
  \begin{enumerate}[label=\alph*)]
  \item $T=\bigcup_\ell F_\ell T$;
  \item there is a weakly monotone sequence $n_\ell$, $\ell\ge0$, and
    maps $x_\ell : \p\D^{n_\ell} \too F_{\ell}T$ and
    $y_\ell : \p\D^{n_\ell} \too F_{\ell-1}T$ such that the following
    diagram is a pushout:
    \begin{equation*}
      \begin{xy}
        \Square[\p\D^{n_\ell}`F_{\ell-1}T`\D^{n_\ell}`F_\ell T;y_\ell```x_\ell]
      \end{xy}
    \end{equation*}
  \end{enumerate}
  The morphism
  \begin{equation*}
    \Map(F_\ell T,X) \to \Map(F_{\ell-1}T\into F_\ell T,f)
  \end{equation*}
  is a cover, since it is a pullback of the cover
  $X_{n_\ell}\to\Map(\p\D^{n_\ell}\into\D^{n_\ell},f)$.
\end{proof}

\begin{definition}
  Let $k$ be a natural number. A simplicial space is a
  \textbf{\boldmath{$k$}-groupoid} if the morphism
  \begin{equation*}
    X_n \too \Map(\L^n_i,X)
  \end{equation*}
  is a cover for all $n>0$ and $0\le i\le n$, and an isomorphism when
  $n>k$. Denote the category of $k$-groupoids by $s_k\CV$.
\end{definition}

\begin{definition}
  A simplicial map $f:X\to Y$ in $s\CV$ is a \textbf{fibration} if the
  morphism
  \begin{equation*}
    X_n \too \Map(\L^n_i\into\D^n,f)
  \end{equation*}
  is a cover for all $n>0$ and $0\le i\le n$.
\end{definition}

Our goal in the remainder of this section is to show that the
$k$-groupoids in a descent category form a category of fibrant
objects.
\begin{theorem}
  \label{main}
  With fibrations and hypercovers as fibrations and trivial
  fibrations, the category of $k$-groupoids $s_k\CV$ is a category of
  fibrant objects.
\end{theorem}

The proof of Theorem~\ref{main} will consist of a sequence of lemmas;
we also take the opportunity to derive some additional useful
properties of fibrations and hypercovers along the way. Axiom
\ref{F:fibrant} is clear.

\begin{definition}
  Let $m>0$. An \textbf{{\boldmath$m$}-expansion} $S\into T$
  (\textbf{expansion}, if $m=1$) is a map of simplicial sets such that
  there exists a filtration
  \begin{equation*}
    S=F_{-1}T \subset F_0T \subset F_1T \subset \dots \subset T
  \end{equation*}
  satisfying the following conditions:
  \begin{enumerate}[label=\alph*)]
  \item $T=\bigcup_\ell F_\ell T$;
  \item there is a weakly monotone sequence $n_\ell\ge m$, $\ell\ge0$,
    a sequence $0\le i_\ell\le n_\ell$, and maps
    $x_\ell : \D^{n_\ell} \too F_{\ell}T$ and
    $y_\ell : \L^{n_\ell}_{i_\ell} \too F_{\ell-1}T$ such that the
    following diagram is a pushout:
    \begin{equation*}
      \begin{xy}
        \Square[\L^{n_\ell}_{i_\ell}`F_{\ell-1}T`\D^{n_\ell}`F_\ell T;y_\ell```x_\ell]
      \end{xy}
    \end{equation*}
  \end{enumerate}
\end{definition}

\begin{lemma}
  \label{proper}
  If $S\subset\D^n$ is the union of $0<m\le n$ faces of the
  $n$-simplex $\D^n$, the inclusion $S\into\D^n$ is an $m$-expansion.
\end{lemma}
\begin{proof}
  The proof is by induction on $n$: the initial step $n=1$ is clear.

  Enumerate the faces of $\D^n$ not in $S$:
  \begin{equation*}
    \{ \p_{i_0}\D^n , \dots , \p_{i_{n-m}}\D^n \} ,
  \end{equation*}
  where $0\le i_0 < \dots < i_{n-m} \le n$. Let
  \begin{equation*}
    F_\ell \D^n = S \cup \bigcup_{j\le\ell} \p_{i_j}\D^n , \quad
    0\le \ell\le n-m .
  \end{equation*}
  By the induction hypothesis, we see that
  $F_{\ell-1}\D^n\cap\p_{i_\ell}\D^n\into\p_{i_\ell}\D^n$ is an
  $m$-expansion: on the one hand, each face of $\D^n$ contained in $S$
  contributes a face of $\p_{i_\ell}\D^n$ to
  $F_{\ell-1}\D^n\cap\p_{i_\ell}\D^n$, and hence
  $F_{\ell-1}\D^n\cap\p_{i_\ell}\D^n$ contains at least $m$ faces of
  $\p_{i_\ell}\D^n$; on the other hand,
  $F_{\ell-1}\D^n\cap\p_{i_\ell}\D^n$ does not contain the face
  $\p_{i_{n-m}}\D^n\cap\p_{i_\ell}\D^n$ of $\p_{i_\ell}\D^n$.
\end{proof}

\begin{lemma}
  \label{expansion}
  Let $T$ be a finite simplicial set, and let $S\into T$ be an
  $m$-expansion.
  \begin{enumerate}[label=\roman*)]
  \item If $X$ is a $k$-groupoid, the induced morphism
    \begin{equation*}
      \Map(T,X) \too \Map(S,X)
    \end{equation*}
    is a cover, and an isomorphism if $m>k$.
  \item If $f:X\to Y$ is a fibration of $k$-groupoids, the induced
    morphism
    \begin{equation*}
      \Map(T,X) \too \Map(S\into T,f)
    \end{equation*}
    is a cover, and an isomorphism if $m>k$.
  \end{enumerate}
\end{lemma}  
\begin{proof}
  The proof is by induction on the length of the filtration of $T$
  exhibiting $S\into T$ as an expansion. In the first case, the
  morphism $\Map(F_\ell T,X)\to\Map(F_{\ell-1}T,X)$ is a cover, since
  it is a pullback of the cover
  $X_{n_\ell}\to\Map(\L^{n_\ell}_{i_\ell},X)$ (which is an isomorphism
  if $m>k$), and in the second case, the morphism
  \begin{equation*}
    \Map(F_\ell T,X)\to\Map(F_{\ell-1}T\into S_\ell,f)
  \end{equation*}
  is a cover, since it is a pullback of the cover
  $X_{n_\ell}\to\Map(\L^{n_\ell}_{i_\ell}\into\D^n,f)$ (which is again an
  isomorphism if $m>k$).
\end{proof}

\begin{corollary}
  \label{face-cover}
  If $X$ is a $k$-groupoid, the face map $\p_i:X_n\to X_{n-1}$ is a
  cover.
\end{corollary}

\begin{lemma}
  \label{fibrations}
  If $f:X\to Y$ is a fibration of $k$-groupoids, then
  \begin{equation*}
    X_n \too \Map(\L^n_i\into\D^n,f)
  \end{equation*}
  is an isomorphism for $n>k$.
\end{lemma}
\begin{proof}
  We have the following commutative diagram, in which the square is a
  pullback:
  \begin{equation*}
    \begin{xy}
      \SQUARE<1000,600>[\Map(\L^n_i\into\D^n,f)`Y_n`\Map(\L^n_i,X)`%
      \Map(\L^n_i,Y);``\gamma`]
      \morphism(-1000,600)<1000,0>[X_n`\Map(\L^n_i\into\D^n,f);\alpha]
      \morphism(-1000,600)|b|<1000,-600>[X_n`\Map(\L^n_i,X);\beta]
    \end{xy}
  \end{equation*}
  If $n>k$ and $0\le i\le n$, $\beta$ and $\gamma$ are isomorphisms,
  and hence $\alpha$ is an isomorphism.
\end{proof}

\begin{lemma}
  \label{fibration}
  A hypercover $f:X\to Y$ of $k$-groupoids is a fibration.
\end{lemma}
\begin{proof}
  For $n>0$ and $0\le i\le n$, we have the following commutative
  diagram, in which the square is a pullback:
  \begin{equation}
    \label{lambda:kstack}
    \begin{aligned}
      \begin{xy}
        \SQUARE<1200,600>[\Map(\p\D^n\into\D^n,f)`X_{n-1}`%
        \Map(\L^n_i\into\D^n,f)`\Map(\p\D^{n-1}\into\D^{n-1},f);``\gamma`\delta]
        \morphism(-1000,600)<1000,0>[X_n`\Map(\p\D^n\into\D^n,f);\alpha]
        \morphism(-1000,600)|b|<1000,-600>[X_n`\Map(\L^n_i\into\D^n,f);\beta]
      \end{xy}
    \end{aligned}
  \end{equation}
  If $n>0$ and $0\le i\le n$, then $\alpha$ and $\gamma$ are covers,
  hence $\beta$ is a cover.
\end{proof}

\begin{lemma}
  Suppose the descent category $\CV$ is subcanonical. If $f:X\to Y$ is
  a hypercover of $k$-groupoids, then
  $X_n \to \Map(\p\D^n\into\D^n,f)$ is an isomorphism for $n\ge k$.
\end{lemma}
\begin{proof}
  Consider the diagram \eqref{lambda:kstack}. If $n>k$, so that
  $\beta$ is an isomorphism, we see that $\alpha$ is both a regular
  epimorphism and a monomorphism, and hence is an isomorphism.

  To handle the remaining case, consider the diagram
  \eqref{lambda:kstack} with $n=k+1$. We have already seen that all
  morphisms in the triangle forming the left side of the diagram are
  isomorphisms. But $\delta$ factors as the composition of the covers
  $\p_i:X_{k+1}\to X_k$ and $\gamma$; hence, it is a cover. Since
  pullback along a cover in $\CV$ reflects isomorphisms, we conclude
  that $\gamma$ is an isomorphism.
\end{proof}

Next, we show that fibrations and hypercovers are closed under
composition.

\begin{lemma}
  \label{compose-Reedy}
  If $f:X\to Y$ and $g:Y\to Z$ are hypercovers, then $g\*f$ is a
  hypercover.
\end{lemma}
\begin{proof}
  Consider the commutative diagram
  \begin{equation}
    \label{basic:1}
    \begin{aligned}
      \begin{xy}
        \SQUARE<1200,600>[\Map(\p\D^n\into\D^n,f)`Y_n`%
        \Map(\p\D^n\into\D^n,g\*f)`\Map(\p\D^n\into\D^n,g);``\gamma`\delta]
        \morphism(-1000,600)<1000,0>[X_n`\Map(\p\D^n\into\D^n,f);\alpha]
        \morphism(-1000,600)|b|<1000,-600>[X_n`\Map(\p\D^n\into\D^n,g\*f);\beta]
      \end{xy}
    \end{aligned}
  \end{equation}
  in which the square is a pullback. Since $\alpha$ and $\gamma$ are
  covers, it follows that $\beta$ is a composition of two covers, and
  hence is itself a cover.
\end{proof}

\begin{lemma}
  \label{compose-fibrations}
  If $f:X\to Y$ and $g:Y\to Z$ are fibrations of $k$-groupoids, then
  $g\*f$ is a fibration.
\end{lemma}
\begin{proof}
  Consider the commutative diagram
  \begin{equation}
    \label{basic:0}
    \begin{aligned}
      \begin{xy}
        \SQUARE<1200,600>[\Map(\L^n_i\into\D^n,f)`Y_n`%
        \Map(\L^n_i\into\D^n,g\*f)`\Map(\L^n_i\into\D^n,g);``\beta`]
        \morphism(-1000,600)<1000,0>[X_n`\Map(\L^n_i\into\D^n,f);]
        \morphism(-1000,600)|b|<1000,-600>[X_n`\Map(\L^n_i\into\D^n,g\*f);\alpha]
      \end{xy}
    \end{aligned}
  \end{equation}
  in which the square is a pullback. If $n>0$ and $0\le i\le n$, then
  $\beta$ is a cover, implying that $\alpha$ is a composition of two
  covers, and hence itself a cover.
\end{proof}

Next, we prove Axioms \ref{F:pullback} and \ref{F:trivial-pullback}.
\begin{lemma}
  \label{pullback-Reedy}
  If $p:X\to Y$ is a hypercover and $f:Z\to Y$ is a morphism, the
  morphism $q$ in the pullback diagram
  \begin{equation*}
    \begin{xy}
      \Square[X\times_YZ`X`Z`Y;`q`p`f]
    \end{xy}
  \end{equation*}
  is a hypercover.
\end{lemma}
\begin{proof}
  In the pullback diagram
  \begin{equation*}
    \begin{xy}
      \Square[X_n\times_{Y_n}Z_n`X_n`\Map(\p\D^n\into\D^n,q)`%
      \Map(\p\D^n\into\D^n,p);`\alpha`\beta`]
    \end{xy}
  \end{equation*}
  the morphism $\alpha$ is a cover because $\beta$ is.
\end{proof}

\begin{lemma}
  If $p:X\to Y$ is a fibration of $k$-groupoids, and $f:Z\to Y$ is a
  morphism of $k$-groupoids, then $X\times_YZ$ is a $k$-groupoid, and
  the morphism $q$ in the pullback diagram
  \begin{equation*}
    \begin{xy}
      \Square[X\times_YZ`X`Z`Y;`q`p`f]
    \end{xy}
  \end{equation*}
  is a fibration.
\end{lemma}
\begin{proof}
  Given $n>0$ and $0\le i\le n$, we have a pullback square
  \begin{equation*}
    \begin{xy}
      \Square[X_n\times_{Y_n}Z_n`X_n`\Map(\L^n_i\into\D^n,q)`%
      \Map(\L_i^n\into\D^n,p);`\alpha`\beta`]
    \end{xy}
  \end{equation*}
  The morphism $\alpha$ is a cover because $\beta$ is.

  There is also a pullback square
  \begin{equation*}
    \begin{xy}
      \Square
      [\Map(\L^n_i\into\D^n,q)`Z_n`\Map(\L^n_i,X\times_YZ)`\Map(\L^n_i,Z);%
      `\gamma``]
   \end{xy}
  \end{equation*}
  If $Z$ is a $k$-groupoid, then $\gamma$ is a cover, and an
  isomorphism if $n>k$.
\end{proof}

Next, we prove that $s\CV$ is a descent category, with hypercovers as
covers: that is, we show that hypercovers satisfy Axiom
\ref{cancellation}.
\begin{lemma}
  \label{groupoid:complete}
  If $f:X\to Y$ and $g:Y\to Z$ are morphisms of simplicial spaces and
  $f$ and $gf$ are hypercovers, then $g$ is a hypercover.
\end{lemma}
\begin{proof}
  In diagram \eqref{basic:1}, $\alpha$ and $\beta$ are covers. We will
  show that $\delta$ is a cover: applying Axiom \ref{cancellation}, it
  follows that $\gamma$ is a cover.

  For $-1\le j\le n-1$, let
  \begin{equation*}
    M_n(f,g,j) = \Map(\sk_j\D^n,X) \times_{\Map(\sk_j\D^n,Y)} \Map(\p^n\D\into\D,g) ,
  \end{equation*}
  where $\sk_j\D^n$, the $j$-skeleton of $\D^n$, is the union of the
  $j$-simplices of $\D^n$. The pullback square
  \begin{equation*}
    \begin{xy}
      \Square[M_n(f,g,j)`\bigl(X_j\bigr)^{\binom{n+1}{j+1}}`%
      M_n(f,g,j-1)`\Map(\p\D^j\into\D,f)^{\binom{n+1}{j+1}};```]
    \end{xy}
  \end{equation*}
  shows that the morphism $M_n(f,g,j)\to M_n(f,g,j-1)$ is a
  cover. Since
  \begin{align*}
    M_n(f,g,-1) &\cong \Map(\p\D^n\into\D^n,g) \intertext{and}
    M_n(f,g,n-1) &\cong \Map(\p\D^n\into\D^n,g\*f) ,
  \end{align*}
  we see that the $\delta$ is a cover.
\end{proof}

In order to show that $k$-groupoids form a category of fibrant
objects, we will need to construct path spaces. In fact, the proof
requires iterated path spaces as well: it is convenient to organize
these into a simplicial functor $P_n$. The proof of Theorem~\ref{main}
actually only requires the functors $P_1$ and $P_2$ (and $P_0$, the
identity functor).
\begin{definition}
  Let $P_n:s\CV\to s\CV$ be the functor on simplicial spaces such that
  \begin{equation*}
    (P_nX)_m = \Map(\D^{m,n},X) ,
  \end{equation*}
  where $\D^{m,n}$ is the prism $\D^m\times\D^n$.
\end{definition}

The functor $P_n$ is the space of maps from the $n$-simplex $\D^n$ to
$X$; in particular, there is a natural isomorphism between $P_0X$ and
$X$, and $PX=P_1X$ is a path space for $X$. Note that $P_n$ preserves
finite limits, and in particular, it preserves the terminal object
$e$. Motivated by Brown's Lemma \ref{P(f)}, we make the following
definition.
\begin{definition}
  A morphism $f:X\to Y$ of $k$-groupoids is a \textbf{weak
    equivalence} if the fibration
  \begin{equation*}
    q(f):P(f)\too Y
  \end{equation*}
  is a hypercover, where $P(f)=X\times_YP_1Y$.
\end{definition}

In the case of Kan complexes, this characterization of weak
equivalences amounts to the vanishing of the relative simplicial
homotopy groups. (A similar approach is taken, in the setting of
simplicial sheaves, by Dugger and Isaksen \cite{di}.)

If $T$ is a finite simplicial set and $X$ is a simplicial space,
denote by $P_TX$ the simplicial space
\begin{equation*}
  (P_TX)_n = \Map(T,P_\bull X_n) \cong \Map(T\times\D^n,X) .
\end{equation*}
The following theorem will be proved in the next section.
\begin{theorem}
  \label{resolution}
  The functor
  \begin{equation*}
    P_\bull : s\CV \too s^2\CV
  \end{equation*}
  satisfies the following properties:
  \begin{enumerate}[label=\alph*)]
  \item \label{Resolution:fibrant} if $n\ge0$ and $f:X\to Y$ is a
    fibration (respectively hypercover), the induced morphism
    \begin{equation*}
      P_nX \too P_{\p\D^n}X \times_{P_{\p\D^n}Y} Y_n 
    \end{equation*}
    is a fibration (respectively hypercover);
  \item if $f:X\to Y$ is a fibration, $n>0$ and $0\le i\le n$, the
    induced morphism
    \begin{equation*}
      P_nX \too P_{\L^n_i}X \times_{P_{\L^n_i}Y} Y_n 
    \end{equation*}
    is a hypercover.
  \end{enumerate}
\end{theorem}

In particular, the functor $P_1$ satisfies the conditions for a
(functorial) path space in a category of fibrant objects: the
simplicial morphism $P_1X\to X\times X$ is a fibration, and the face
maps $P_1X\to X$ are hypercovers. Lemma \ref{P(f)} now implies the
following.
\begin{lemma}
  Axiom \ref{factorize} holds in $s_k\CV$.
\end{lemma}

\begin{lemma}
  The weak equivalences form a subcategory of $s_k\CV$.
\end{lemma}
\begin{proof}
  Let $f:X\to Y$ and $g:Y\to Z$ be weak equivalences in $s_k\CV$. Form
  the pullback
  \begin{equation*}
    \begin{xy}
      \SQUARE/>`>`>`/<2000,500>[P(g,f)`P_2Z`P(f)\cong X\times_YP_1Y`P_1Z;%
      ``\p_0`]
      \morphism|b|<1200,0>[P(f)\cong X\times_YP_1Y`P_1Y;f\times_YP_1Y]
      \morphism(1200,0)|b|<800,0>[P_1Y`P_1Z;P_1g]
    \end{xy}
  \end{equation*}
  In the following commutative diagram, the solid arrows are
  hypercovers:
  \begin{equation*}
    \begin{xy}
      \SQUARE/{->}`{->}`{}`{->}/<1200,600>%
      [P(g,f)`P(g\*f)\times_XP(f)`P(g)\times_YP(f)`P(g);```]
      \SQUARE(1200,0)/{->}`{}`{-->}`{->}/<1200,600>%
      [P(g\*f)\times_XP(f)`P(g\*f)`P(g)`Z;```]
    \end{xy}
  \end{equation*}
  The result now follows from Lemma~\ref{groupoid:complete}.
\end{proof}

\begin{lemma}
  If $f:X\to Y$ and $g:Y\to Z$ are morphisms of $k$-groupoids such
  that $f$ and $gf$ are weak equivalences, then $g$ is a weak
  equivalence.
\end{lemma}
\begin{proof}
  In the following commutative diagram, the solid arrows are hypercovers:
  \begin{equation*}
    \begin{xy}
      \SQUARE/{->}`{->}`{}`{->}/<1200,600>%
      [P(g,f)`P(g\*f)\times_XP(f)`P(g)\times_YP(f)`P(g);```]
      \SQUARE(1200,0)/{->}`{}`{->}`{-->}/<1200,600>%
      [P(g\*f)\times_XP(f)`P(g\*f)`P(g)`Z;```]
    \end{xy}
  \end{equation*}
  Again, the result follows from Lemma~\ref{groupoid:complete}.
\end{proof}

\begin{lemma}
  A fibration $f:X\to Y$ of $k$-groupoids is a weak equivalence if and
  only if it is a hypercover.
\end{lemma}
\begin{proof}
  In the following commutative diagram, the solid arrows are
  hypercovers:
  \begin{equation*}
    \begin{xy}
      \SQUARE/{->}`{->}`{-->}`{-->}/<800,500>[P_1X`P(f)`X`Y;``q(f)`f]
    \end{xy}
  \end{equation*}
  It follows by Lemma \ref{groupoid:complete} that $f$ is a hypercover
  if and only if $q(f)$ is.
\end{proof}

In order to complete the proof that $s_k\CV$ is a category with weak
equivalences, we need the following result, which is familiar in the
case where $\CV$ is a topos.
\begin{lemma}
  \label{Hard}
  If $f:X\to Y$ is a fibration of $k$-groupoids, and $g:Y\to Z$ and
  $g\*f$ are hypercovers, then $f$ is a hypercover.
\end{lemma}
\begin{proof}
  The idea is to use the fact that $X_{n+1}\to\L_{n+1,1}(f)$ is a
  cover in $\CV$ in order to show that $X_n\to\Map(\p\D^n\into\D,f)$
  is a cover.

  Define the fibred products
  \begin{equation*}
    \begin{xy}
      \SQUARE/>`>`>`/<1800,500>%
      [V(f,g)`X_{n+1}`X_n`\Map(\L^{n+1}_1\into\D^n,gf); a`b``]
      \morphism|b|<800,0>[X_n`X_{n+1};s_0]
      \morphism(800,0)|b|<1000,0>[X_{n+1}`\Map(\L^{n+1}_1\into\D^n,gf);`]
    \end{xy}
  \end{equation*}
  \begin{equation*}
    \begin{xy}
      \SQUARE/>`>`>`/<1800,500>%
      [W(f,g)`X_{n+1}`X_n`\Map(\L^{n+1}_0\into\D^n,gf);\tilde{a}`\tilde{b}``]
      \morphism|b|<800,0>[X_n`X_{n+1};s_0]
      \morphism(800,0)|b|<1000,0>[X_{n+1}`\Map(\L^{n+1}_0\into\D^n,gf);`]
    \end{xy}
  \end{equation*}
  The spaces $V(f,g)$ and $W(f,g)$ are isomorphic: there is a morphism
  from $V(f,g)$ to $W(f,g)$, defined by the diagram
  \begin{equation*}
    \begin{xy}
      \morphism(-700,1000)<2500,-500>[V(f,g)`X_{n+1};a]
      \morphism(-700,1000)|b|<700,-1000>[V(f,g)`X_n;\p_0\*a
]      \morphism(-700,1000)/.>/<700,-500>[V(f,g)`W(f,g);]
      \SQUARE/>`>`>`/<1800,500>%
      [W(f,g)`X_{n+1}`X_n`\Map(\L^{n+1}_0\into\D^n,gf);```]
      \morphism|b|<800,0>[X_n`X_{n+1};s_0]
      \morphism(800,0)|b|<1000,0>[X_{n+1}`\Map(\L^{n+1}_0\into\D^n,gf);`]
    \end{xy}
  \end{equation*}
  Likewise, there is a morphism from $W(f,g)$ to $V(f,g)$, induced by
  the morphisms $\tilde{a}:V(f,g)\to X_{n+1}$ and
  $\p_1\*\tilde{a}:V(f,g)\to X_n$.  These morphisms between $V(f,g)$
  and $W(f,g)$ are inverse to each other.

  In this way, we see that the morphism $\p_0\*a:V_n(f,g)\to X_n$ is a
  cover: under the isomorphism $V(f,g)\cong W(f,g)$, it is identified
  with the morphism $\tilde{b}:V(f,g)\to X_n$, and this map is a
  pullback of a cover by Lemma~\ref{cofibration}, since $gf$ is a
  hypercover.

  Define the additional fibred products
  \begin{equation*}
    \begin{xy}
      \SQUARE/>`>`>`/<1800,500>[T(f,g)`Y_n`X_n`\Map(\p\D^n\into\D^n,g);```]
      \morphism(0,0)<800,0>[X_n`Y_n;]
      \morphism(800,0)<1000,0>[Y_n`\Map(\p\D^n\into\D^n,g);]
    \end{xy}
  \end{equation*}
  \begin{equation*}
    \begin{xy}
      \SQUARE/>`>`>`/<2400,500>%
      [U(f,g)`Y_{n+1}`X_n`\Map(\L^{n+1}_0\into\D^n,g);```]
      \morphism(0,0)|b|<700,0>[X_n`X_{n+1};s_0]
      \morphism(700,0)<700,0>[X_{n+1}`Y_{n+1};]
      \morphism(1400,0)<1000,0>[Y_{n+1}`\Map(\L^{n+1}_0\into\D^n,g);]
    \end{xy}
  \end{equation*}
  We have the following morphisms between the spaces $T(f,g)$,
  $U(f,g)$, and $V(f,g)$, each of which is a cover:
  \begin{equation*}
    \begin{xy}
      \SQUARE|aaaa|/{}`{=}`{=}`{->}/<1730,250>%
      [T(f,g)`\Map(\p\D^n\into\D^n,f)`X_n\times_{\Map(\p\D^n\into\D^n,g)}Y_n`%
      \Map(\p\D^n\into\D^n,gf)\times_{\Map(\p\D^n\into\D^n,g)}Y_n;```]
      \SQUARE(0,-550)|aaaa|/{}`{=}`{=}`{->}/<1730,250>[U(f,g)`T(f,g)`%
      X_n\times_{\Map(\L^{n+1}_1\into\D^n,g)}Y_{n+1}`%
      X_n\times_{\Map(\L^{n+1}_1\into\D^n,g)}\Map(\p\D^{n+1}\into\D^n,g);```]
      \SQUARE(0,-1100)|aaaa|/{}`{=}`{=}`{->}/<1730,250>[V(f,g)`U(f,g)`%
      X_n\times_{\Map(\L^{n+1}_1\into\D^n,g\*f)}X_{n+1}`%
      X_n\times_{\Map(\L^{n+1}_1\into\D^n,g\*f)}\Map(\L^{n+1}_1\into\D^n,f);```]
    \end{xy}
  \end{equation*}
  In this way, we obtain a diagram
  \begin{equation*}
    \begin{xy}
      \Atriangle/{<-}`{->}`{-->}/<1200,300>%
      [V(f,g)`X_n`\Map(\p\D^n\into\D^n,f);``]
    \end{xy}
  \end{equation*}
  in which the solids arrows are covers, and hence the third arrow is
  as well.
\end{proof}

We can now complete the proof of Theorem~\ref{main}.
\begin{lemma}
  If $f:X\to Y$ and $g:Y\to Z$ are morphisms of $k$-groupoids such
  that $g$ and $gf$ are weak equivalences, then $f$ is a weak
  equivalence.
\end{lemma}
\begin{proof}
  In the following commutative diagram, the solid arrows are
  hypercovers, while the dashed arrow is a fibration:
  \begin{equation*}
    \begin{xy}
      \SQUARE/{->}`{->}`{->}`{}/<1500,600>%
      [P(g,f)`P(g\*f)\times_XP(f)`P(g)\times_YP(f)`P(g\*f);``
      P(g\*f)\times_Xp(f)`]
      \SQUARE(0,-600)/{}`{-->}`{->}`{->}/<1500,600>%
      [P(g)\times_YP(f)`P(g\*f)`P(g)`Z;%
      `P(g)\times_Yp(f)`q(g\*f)`q(g)]
    \end{xy}
  \end{equation*}
  It follows by Lemma~\ref{groupoid:complete} that the composition
  \begin{equation*}
    \begin{xy}
      \morphism/-->/<1200,0>[P(g)\times_YP(f)`P(g);P(g)\times_Yq(f)]
      \morphism(1200,0)<800,0>[P(g)`Z;q(g)]
    \end{xy}
  \end{equation*}
  is a hypercover. Lemma~\ref{Hard} implies that $P(g)\times_Yq(f)$ is
  a hypercover. In the following commutative diagram, the solid arrows
  are hypercovers, while the dashed arrow is a fibration:
  \begin{equation*}
    \begin{xy}
      \Square/{->}`{->}`{-->}`{->}/%
      [P(g)\times_YP(f)`P(f)`P(g)`Y;p(g)\times_YP(f)`P(g)\times_Yq(f)`q(f)`p(g)]
    \end{xy}
  \end{equation*}
  Applying Lemma~\ref{groupoid:complete} one final time, we conclude
  that $q(f)$ is a hypercover, and hence that $f$ is a weak
  equivalence.
\end{proof}

\section{The simplicial resolution for $k$-groupoids}

In this section, we prove Theorem \ref{resolution}.
Consider the following subcomplexes of the prism $\D^{m,n}$:
\begin{align*}
  \L^{m,n}_i &= (\L^m_i\times\D^n)\cup(\D^m\times\p\D^n) &                         \tilde{\L}^{m,n}_j &= (\p\D^m\times\D^n)\cup(\D^m\times\L^n_j) .
\end{align*}
Moore has proved that the inclusions $\L^{m,n}_i\into\D^{m,n}$ and
$\tilde{\L}^{m,n}_j\into\D^{m,n}$ are expansions. The following lemma
is a refinement of his theorem.
\begin{lemma}
  \label{Moore}
  The inclusions $\L^{m,n}_i\into\D^{m,n}$ and
  $\tilde{\L}^{m,n}_j\into\D^{m,n}$ are $m$- and $n$-expansions
  respectively.
\end{lemma}
\begin{proof}
  The proof is a modification of an argument of Cartan
  \cite{cartan}. The proofs of the two parts are formally identical,
  and we will concentrate on the former.

  An $(m,n)$-shuffle is a permutation $\pi$ of $\{1,\dots,m+n\}$ such
  that
  \begin{equation*}
    \text{$\pi(1)<\dots<\pi(m)$ and $\pi(m+1)<\dots<\pi(m+n)$.}
  \end{equation*}
  The $(m,n)$-shuffles index the $\binom{m+n}{m}$ non-degenerate
  simplices of the prism $\D^{m,n}$: we denote the simplex labeled by
  a shuffle $\pi$ by the same symbol $\pi$. Any simplex of dimension
  $m+n-1$ in $\D^{m,n}$ lies in at most two top-dimensional simplices.

  The geometric realization of the simplex $\D^n$ is the convex
  hull of the vertices
  \begin{equation*}
    v_i = ( \underbrace{0,\dots,0}_{\text{$n-i$ times}} ,
    \underbrace{1,\dots,1}_{\text{$i$ times}} ) \in \R^n .
  \end{equation*}
  Thus, the simplex is the convex set
  \begin{equation*}
    \D^n = \{ (t_1,\dots,t_n) \subset \R^n \mid 0 \le t_1 \le \dots
    \le t_n \le 1 \} .
  \end{equation*}
  Given sequences $0<s_1\dots<s_m<1$ and $0<t_1<\dots<t_n<1$ such that
  $s_i\ne t_j$, representing a pair of points in the interiors of
  $\D^m$ and $\D^n$ respectively, the union of these sequences
  determines a word of length $m+n$ in the letters $s$ and $t$, with
  $m$ letters $s$ and $n$ letters $t$, and hence an
  $(m,n)$-shuffle. The set of such points associated to a shuffle
  $\pi$ is the interior of the geometric realization
  $|\pi|\subset|\D^{m,n}|\cong|\D^m|\times|\D^n|$.

  Represent an $(m,n)$-shuffle $\pi$ by the sequence of natural
  numbers
  \begin{equation*}
    0 \le a_1(\pi) \le \dots \le a_m(\pi) \le n ,
  \end{equation*}
  in such a way that the associated shuffle has the form
  \begin{equation*}
    t^{a_1}st^{a_2-a_1}s\dots t^{a_m-a_{m-1}}st^{n-a_m} ,
  \end{equation*}
  in other words,
  \begin{equation*}
    0 = s_0 < \cdots < s_j < t_{a_j+1} < \cdots < t_{a_{j+1}} <
    s_{j+1} < \cdots s_{m+1} = 1 .
  \end{equation*}
  We adopt the convention that $a_0=0$ and $a_{m+1}=n$.

  Filter $\D^{m,n}$ by the subcomplexes
  \begin{equation*}
    F_\ell\D^{m,n}=\L^{m,n}_i\cup\bigcup_{\{\pi\mid b(\pi,i)\le \ell \}} \pi ,
  \end{equation*}
  where
  \begin{equation*}
    b(\pi,i) = \sum_{j=1}^i a_j(\pi) - \sum_{j=i+1}^m a_j(\pi) .
  \end{equation*}
  The faces of a top-dimensional simplex $\pi$ are as follows:
  \begin{itemize}
  \item the geometric realization of the face $\p_{a_j+j-1}\pi$ is the
    intersection of the geometric realization of the simplex $\pi$
    with the hyperplane
    \begin{equation*}
      t_{a_j} = s_j ,
    \end{equation*}
    when $a_{j-1}<a_j$, and the hyperplane
    \begin{equation*}
      s_{j-1} = s_j ,
    \end{equation*}
    when $a_{j-1}=a_j$;
  \item the geometric realization of the face $\p_{a_j+j}\pi$ is the
    intersection of the geometric realization of the simplex $\pi$
    with the hyperplane
    \begin{equation*}
      s_j = t_{a_j+1} ,
    \end{equation*}
    when $a_j<a_{j+1}$, and the hyperplane
    \begin{equation*}
      s_j = s_{j+1} ,
    \end{equation*}
    when $a_j=a_{j+1}$;
  \item when $a_j+j<k<a_{j+1}+j$, the geometric realization of the
    face $\p_k\pi$ is the intersection of the geometric realization of
    the simplex $\pi$ with the hyperplane
    \begin{equation*}
      t_{k-j} = t_{k-j+1} .
    \end{equation*}
  \end{itemize}
  We must show that at least one face of $\pi$ does not lie in
  $F_{b(\pi,i)-1}\D^{m,n}$:
  \begin{enumerate}[label=\roman*)]
  \item if $a_i(\pi)=a_{i+1}(\pi)$, the face $\p_{a_i+i}\pi$ is not
    contained in $\L^{m,n}_i$, nor in any top-dimensional simplex of
    $\D^{m,n}$ other than $\pi$;
  \item if $a_i(\pi)<a_{i+1}(\pi)$ and $i>0$, the face
    $\p_{a_i+i}\pi$ is contained in the simplex $\tilde{\pi}$ with
    \begin{equation*}
      a_j(\tilde{\pi}) =
      \begin{cases}
        a_j(\pi) , & j<i , \\
        a_j(\pi)+1 , & j=i , \\
        a_j(\pi) , & j>i ,
      \end{cases}
    \end{equation*}
    for which $b(\tilde{\pi},i)=b(\pi,i)+1$;
  \item if $a_i(\pi)<a_{i+1}(\pi)$ and $i<m$, the face
    $\p_{a_{i+1}+i-1}\pi$ is contained in the simplex $\tilde{\pi}$
    with
    \begin{equation*}
      a_j(\tilde{\pi}) =
      \begin{cases}
        a_j(\pi) , & j<i+1 , \\
        a_j(\pi)-1 , & j=i+1 , \\
        a_j(\pi) , & j>i+1 ,
      \end{cases}
    \end{equation*}
    for which $b(\tilde{\pi},i)=b(\pi,i)+1$.
  \end{enumerate}
  By Lemma~\ref{proper}, the proof is completed by enumerating at
  least $m$ faces of $\pi$ which lie in either $\L^{m,n}_i$ or a
  simplex $\tilde{\pi}$ for which $b(\tilde{\pi},i)=b(\pi,i)-1$:
  \begin{enumerate}[label=\roman*)]
  \item For each $j<i$ with $a_j<a_{j+1}$, we obtain $a_{j+1}-a_j$
    such faces as follows:
    \begin{enumerate}[label=a\arabic*)]
    \item the $a_{j+1}-a_j-1$ faces $\p_\ell\pi$ with
      $a_j+j<\ell<a_{j+1}+j-1$ lie in $\L^{m,n}_i$;
    \item the face $\p_{a_{j+1}+j-1}\pi$ lies in the simplex
      $\tilde{\pi}$ with
      \begin{equation*}
        a_j(\tilde{\pi}) =
        \begin{cases}
          a_k(\pi) , & k<j+1 , \\
          a_k(\pi)-1 , & k=j+1 , \\
          a_k(\pi) , & k>j+1 ,
        \end{cases}
      \end{equation*}
      for which $b(\tilde{\pi},i)=b(\pi,i)-1$.
    \end{enumerate}
  \item For each $j>i$ with $a_j<a_{j+1}$, we obtain $a_{j+1}-a_j$
    such faces as follows:
    \begin{enumerate}[label=b\arabic*)]
    \item the $a_{j+1}-a_j-1$ faces $\p_\ell\pi$ with
      $a_j+j+1<\ell<a_{j+1}+j$ lie in $\L^{m,n}_i$;
    \item the face $\p_{a_j+j+1}\pi$ lies in the simplex $\tilde{\pi}$
      with
      \begin{equation*}
        a_j(\tilde{\pi}) =
        \begin{cases}
          a_k(\pi) , & k<j , \\
          a_k(\pi)+1 , & k=j , \\
          a_k(\pi) , & k>j ,
        \end{cases}
      \end{equation*}
      for which $b(\tilde{\pi},i)=b(\pi,i)-1$.
    \end{enumerate}
  \item The $a_{i+1}-a_i-1$ faces $\p_\ell\pi$ with
    $a_i+i<\ell<a_{i+1}+i-1$ lie in $\L^{m,n}_i$.
  \item The face $\p_0\pi$ lies in $\L^{m,n}_i$ unless $i=0$ and
    $a_1=0$.
  \item The face $\p_{m+n}\pi$ lies in $\L^{m,n}_i$ unless $i=m$
    and $a_m=n$. \qedhere
  \end{enumerate}
\end{proof}

\begin{lemma}
  \label{Moore2}
  Let $T$ be a finite simplicial set, and let $S\into T$ be
  a simplicial subset. Then
  \begin{equation*}
    \D^m\times S \cup \L^m_i\times T \into \D^m\times T
  \end{equation*}
  is an $m$-expansion, and
  \begin{equation*}
    S \times \D^n \cup T \times \L^n_j \into T \times \D^n
  \end{equation*}
  is an $n$-expansion.
\end{lemma}
\begin{proof}
  We prove the first statement: the proof of the second is analogous.

  Filter $T$ by the simplicial subsets $F_\ell T=S\cup\sk_\ell T$. Let
  $I_\ell$ be the set of nondegenerate simplices in
  $T_\ell\setminus S_\ell$. There is a pushout square
  \begin{equation*}
    \begin{xy}
      \Square[(\L^{m,\ell}_i)^{I_\ell}`%
      \D^m\times F_{\ell-1}T \cup \L^m_i\times T`%
      (\D^{m,\ell})^{I_\ell}`\D^m\times F_\ell T \cup \L^m_i\times T;```]
    \end{xy}
  \end{equation*}
  and by Lemma~\ref{Moore}, the vertical arrows of this diagram are
  $m$-expansions. Composing the $m$-expansions
  \begin{equation*}
    \D^m\times F_{\ell-1}T \cup \L^m_i\times T \into \D^m\times F_\ell T
    \cup \L^m_i\times T
  \end{equation*}
  for $\ell\ge0$, we obtain the result.
\end{proof}

\begin{proof}[Proof of Theorem~\ref{resolution}]
  Let $X$ be a $k$-groupoid. To show that $P_nX$ is a $k$-groupoid, we
  must show that for all $0\le i\le m$, the morphism
  \begin{equation*}
    P_nX_m \too \Map(\L^m_i,P_nX)
  \end{equation*}
  is a cover, if $m>0$, and an isomorphism, if $m>k$. This follows by
  Part~i) of Lemma~\ref{expansion}, since $\L^{m,n}_i\into\D^{m,n}$ is
  an $m$-expansion.

  If $f:X\to Y$ is a fibration, then for all $n\ge0$, the simplicial
  morphism
  \begin{equation*}
    P_nX \too \Map(\p\D^n,P_\bull X) \times_{\Map(\p\D^n,P_\bull Y)} P_nY
  \end{equation*}
  is a fibration since for all $m>0$, the morphism
  $\L^{m,n}_i \into \D^{m,n}$ is an expansion, and for all $n>0$, the
  simplicial morphism
  \begin{equation*}
    P_nX \too \Map(\L^n_j,P_\bull X) \times_{\Map(\L^n_j,P_\bull Y)} P_nY
  \end{equation*}
  is a cover since for all $m>0$, the morphism
  $\tilde{\L}^{m,n}_j \into \D^{m,n}$ is an expansion.

  If $f:X\to Y$ is a hypercover, then for all $n\ge0$, the simplicial
  morphism
  \begin{equation*}
    P_nX \too \Map(\p\D^n,P_\bull X) \times_{\Map(\p\D^n,P_\bull Y)} P_nY
  \end{equation*}
  is a cover, by Lemma \ref{cofibration} applied to the inclusion of
  simplicial sets
  \begin{equation*}
    (\p\D^m\times\D^n)\cup(\D^m\times\D^n) \into \D^{m,n} .
    \qedhere
  \end{equation*}
\end{proof}

\section{A characterization of weak equivalences between
  $k$-groupoids}

A morphism $f:X\to Y$ of $k$-groupoids is a weak equivalence if and
only if the morphism
\begin{equation*}
  P(f)_n \too \Map(\p\D^n\into\D^n,q(f))
\end{equation*}
is a cover for $n\ge0$. When $n=0$, this condition says that the
morphism
\begin{equation*}
  X_0\times_{Y_0}Y_1\to Y_0
\end{equation*}
is a cover, which is a translation to the setting of simplicial spaces
of the condition for a morphism between Kan complexes that the induced
morphism of components $\pi_0(f):\pi_0(X)\to\pi_0(Y)$ be
surjective. For $n>0$, it analogous to the condition for a morphism of
Kan complexes $f:X\to Y$ that the relative homotopy groups
$\pi_{n+1}(Y,X)$ (with arbitrary choice of basepoint) vanish.

The following theorem is analogous to Gabriel and Zisman's famous
theorem on anodyne extensions \cite{GZ}*{Chapter IV, Section 2}.
\begin{theorem}
  \label{We}
  A morphism $f:X\to Y$ of $k$-groupoids is a weak equivalence if and
  only if the morphisms
  \begin{equation}
    \label{we}
    \Map(\D^n\into\D^{n+1},f) \too \Map(\p\D^n\into\L^{n+1}_{n+1},f)
  \end{equation}
  are covers for $n\ge0$.
\end{theorem}
\begin{proof}
  We have
  \begin{equation*}
    P(f)_n \cong \Map(\D^n\into\D^{1,n},f) ,
  \end{equation*}
  and
  \begin{equation*}
    \Map(\p\D^n\into\D^n,q(f)) \cong \Map(\p\D^n\into\L^{1,n}_1,f) .
  \end{equation*}
  This shows that $f$ is a weak equivalence if and only if the
  morphisms
  \begin{equation}
    \label{we*}
    \Map(\D^n\into\D^{1,n},f) \too \Map(\p\D^n\into\L^{1,n}_1,f)
  \end{equation}
  are covers for all $n\ge0$.

  Suppose that the morphism \eqref{we} is a cover for $n\ge0$; we show
  that \eqref{we*} is a cover for $n\ge0$. For $0\le i\le n$, let
  $\D^{n+1}_i\subset\D^{1,n}$ be the simplex whose vertices
  are
  \begin{equation*}
    \{ (0,0) , \dots , (0,i) , (1,i) , \dots , (1,n) \} .
  \end{equation*}
  Observe that
  \begin{equation*}
    \D^{n+1}_{i-1} \cap \D^{n+1}_i = \p_i\D^{n+1}_{i-1} =
    \p_i\D^{n+1}_i .
  \end{equation*}

  Filter the prism:
  \begin{equation*}
    F_i\D^{1,n} = \L^{1,n}_1 \cup \D^{n+1}_0 \cup \dots
    \cup \D^{n+1}_i .
  \end{equation*}
  If $i<n$, there is a pullback diagram
  \begin{equation*}
    \begin{xy}
      \SQUARE<1500,500>[\Map(\p\D^n\into F_i\D^{1,n},f)`%
      Y_{n+1}`\Map(\p\D^n\into F_{i-1}\D^{1,n},f)`Y_n;``\p_i`]
    \end{xy}
  \end{equation*}
  The vertical morphisms are covers by part i) of
  Lemma~\ref{expansion}: composing them for $0\le i<n$, we see that
  the morphism
  \begin{equation*}
    \Map(\p\D^n\into F_{n-1}\D^{1,n},f) \too
    \Map(\p\D^n\into\L^{1,n}_1,f)
  \end{equation*}
  is a cover.

  There is also a pullback diagram
  \begin{equation*}
    \begin{xy}
      \Square[\Map(\D^n\into\D^{1,n},f)`%
      \Map(\D^n\into\D^{n+1},f)`%
      \Map(\p\D^n\into F_{n-1}\D^{1,n},f)`%
      \Map(\p\D^n\into\L^{n+1}_{n+1},f);```]
    \end{xy}
  \end{equation*}
  The right-hand vertical morphism is a cover by hypothesis, and hence
  the left-hand vertical morphism, namely \eqref{we*}, is also a
  cover.

  Now, suppose that \eqref{we*} is a cover for $n\ge0$; we show that
  \eqref{we} is a cover for $n\ge0$. There is a map from $\D^{1,n}$ to
  $\D^{n+1}$, which takes the vertex $(0,i)$ to $i$, and the vertices
  $(1,i)$ to $n+1$. This map takes the simplicial subset
  $\L^{1,n}_1\subset\D^{1,n}$ to the horn
  $\L^{n+1}_{n+1}\subset\D^{n+1}$, and induces a pullback square
  \begin{equation*}
    \begin{xy}
      \Square[\Map(\D^n\into\D^{n+1},f)`%
      \Map(\D^n\into\D^{1,n},f)`%
      \Map(\p\D^n\into\L^{n+1}_{n+1},f)`%
      \Map(\p\D^n\into\L^{1,n}_1,f);```]
    \end{xy}
  \end{equation*}
  It follows that \eqref{we} is a cover for $n\ge0$.
\end{proof}

\section{$k$-categories}

In this section, we study a class of simplicial spaces bearing the
same relationship to $k$-groupoids as categories bear to groupoids.
The definition of $k$-categories is inspired by Rezk's definition of a
complete Segal space \cite{Rezk}.

Recall that the thick $1$-simplex $\DELTA^1$ is the nerve of the
groupoid $\[1\]$ with objects $\{0,1\}$ and a single morphism between
any pair of objects.
\begin{definition}
  Let $k>0$. A \textbf{\boldmath{$k$}-category} in a descent category
  $\CV$ is a simplicial space $X$ such that
  \begin{enumerate}[label=\arabic*)]
  \item if $0<i<n$, the morphism
    \begin{equation*}
      X_n \to \Map(\L^n_i,X)
    \end{equation*}
    is a cover, and an isomorphism if $n>k$;
  \item if $i\in\{0,1\}$, the morphism
    \begin{equation*}
      \Map(\DELTA^1,X) \to \Map(\L^1_i,X) \cong X_0
    \end{equation*}
    is a cover.
  \end{enumerate}
\end{definition}

There is an involution permuting the two vertices of $\DELTA^1$. Thus,
in the second axiom above, it suffices to consider one of the two
morphisms $\Map(\DELTA^1,X) \to \Map(\L^1_i,X)$, since they are
isomorphic.

\begin{lemma}
  \label{coskeletal}
  A $k$-category $X$ is $k+1$-coskeletal, that is, for every $n\ge0$,
  \begin{equation*}
    X_n \cong \cosk_{k+1}X_n = \Map(\sk_{k+1}\D^n,X) .
  \end{equation*}
\end{lemma}
\begin{proof}
  Consider the commutative diagram
  \begin{equation*}
    \begin{xy}
      \Atriangle/{<-}`>`>/<1200,300>[\Map(\p\D^n,X)`X_n`\Map(\L^n_{n-1},X);%
      \alpha_n`\beta_n`\gamma_n]
    \end{xy}
  \end{equation*}
  If $n>k$, the morphism $\gamma_n$ is an isomorphism, and hence
  $\beta_n$ is a split epimorphism.

  If furthermore $n>k+1$, the upper horizontal morphism of the
  pullback square
  \begin{equation*}
    \begin{xy}
      \Square[\Map(\p\D^n,X)`X_{n-1}`\Map(\L^n_{n-1},X)`%
      \Map(\p\D^{n-1},X);`\beta_n`\alpha_{n-1}`\p_{n-1}]
    \end{xy}
  \end{equation*}
  factors into a composition
  \begin{equation*}
    \Map(\p\D^n,X) \xrightarrow{\beta_n} \Map(\L^n_{n-1},X)
    \xrightarrow{\beta_{n-1}\p_{n-1}} \Map(\L^{n-1}_{n-2},X)
    \xrightarrow{\gamma_{n-1}^{-1}} X_{n-1}
  \end{equation*}
  and hence, by universality of the pullback square, the morphism
  $\beta_n$ is a monomorphism. Since this morphism is also a split
  epimorphism, it follows that $\beta_n$ is an isomorphism. We
  conclude that $\alpha_n$ is an isomorphism.

  The pullback square
  \begin{equation*}
    \begin{xy}
      \Square[\cosk_\ell X_n`\bigl(X_\ell\bigr)^{\binom{n+1}{\ell+1}}`%
      \cosk_{\ell-1}X_n`\Map(\p\D^\ell,X)^{\binom{n+1}{\ell+1}};%
      ``\alpha_\ell^{\binom{n+1}{\ell+1}}`]
    \end{xy}
  \end{equation*}
  shows that the morphism $\cosk_\ell X_n\cosk_{\ell-1}X_n$ is an
  isomorphism if $\ell>k+1$. The lemma follows by downward induction
  in $\ell$, since $X_n\cong\cosk_nX_n$.
\end{proof}

If $T$ is a finite simplicial set, form the coend
\begin{equation*}
  T\times_\D\DELTA = \textstyle\int^{n\in\D} T_n \times \DELTA^n .
\end{equation*}
(This is denoted $k_!T$ by Joyal and Tierney \cite{JT}.) As examples
of this construction, we have the thick horns
\begin{equation*}
  \LAMBDA^n_i = \L^n_i\times_\D\DELTA \subset \DELTA^n
\end{equation*}
and the thick boundary
\begin{equation*}
  \p\DELTA^n = \D^n\times_\D\DELTA \subset \DELTA^n
\end{equation*}
Of course, $\LAMBDA^1_i\cong\L^1_i$, and $\p\DELTA^1\cong\p\D^1$.
 
Inner expansions play the same role in the theory of
$k$-categories that expansions play in the theory of $k$-groupoids.
\begin{definition}
  An \textbf{\boldmath inner $m$-expansion} (inner expansion, if
  $m=1$) is a map of simplicial sets such that there exists a
  filtration
  \begin{equation*}
    S = F_{-1}T \subset F_0T \subset F_1T \subset \dots \subset T
  \end{equation*}
  satisfying the following conditions:
  \begin{enumerate}[label=\arabic*)]
  \item $T=\bigcup_\ell F_\ell T$;
  \item there is a weakly monotone sequence $n_\ell\ge m$, a sequence
    $0<i_\ell<n_\ell$, and maps $x_\ell : \D^{n_\ell} \too F_\ell T$
    and $y_\ell : \L^{n_\ell}_{i_\ell} \too F_{\ell-1}T$ such that the
    following diagram is a pushout:
    \begin{equation*}
      \begin{xy}
        \SQUARE<600,400>[\L^{n_\ell}_{i_\ell}`F_{\ell-1}T`\D^{n_\ell}`%
        F_\ell T;y_\ell```x_\ell]
      \end{xy}
    \end{equation*}
  \end{enumerate}
\end{definition}

It is not hard to see that inner $n$-expansions form a category.

\begin{lemma}
  \label{lambda}
  If $0<i<n$, the inclusion $\LAMBDA^n_i\cup\D^n \into \DELTA^n$ is an
  inner $n$-expansion.
\end{lemma}
\begin{proof}
  The $k$-simplices of $\DELTA^n$ have the form $(i_0,\dots,i_k)$,
  where $i_0,\dots,i_k\in\{0,\dots,n\}$; a $k$-simplex is
  nondegenerate if $i_{j-1}\ne i_j$ for $1\le j\le k$.

  Let $Q_{k,m}$, $0\le m<k-i$ be the set of non-degenerate $k$-simplices
  $s=(i_0\dots i_k)$ of $\DELTA^n$ which satisfy the following
  conditions:
  \begin{enumerate}[label=\alph*)]
  \item $s$ is not contained in $\LAMBDA^n_i\cup\D^n$;
  \item $i_{j-1}=i_{j+1}$ for $i\le j<i+m$;
  \item $i_{i+m}=i$;
  \item $i_{i+m-1}\ne i_{i+m+1}$.
  \end{enumerate}
  For example, if $n=2$ and $i=1$, then $Q_{2,0}=\{(2,1,0)\}$,
  \begin{equation*}
    Q_{3,1}=\{(1,0,1,2),(1,2,1,0)\} ,
  \end{equation*}
  and
  \begin{equation*}
    Q_{3,0}=\{(0,1,2,0),(0,1,2,1),(2,1,0,1),(2,1,0,2)\} .
  \end{equation*}
  Let $R_k$ be the set of non-degenerate $k$-simplices which do not
  lie in $\LAMBDA^n_i\cup\D^n$, nor in any of the sets $Q_{k,m}$.

  The simplicial set $\DELTA^n$ is obtained from $\LAMBDA^n_i\cup\D^n$
  by inner expansions along the simplices of type $Q_{k,m}$ in order
  first of increasing $k$, then of decreasing $m$. (The order in which
  the simplices are adjoined within the sets $Q_{k,m}$ is
  unimportant.)

  To prove this, consider a simplex $s=(i_0,\dots,i_k)$ in
  $R_k$. There is a unique natural number $0\le m_s<k-i$ such that the
  simplex
  \begin{equation*}
    \ts = (i_0,\dots,i_{i+m_s-1},i,i_{i+m_s},\dots,i_k)
  \end{equation*}
  has type $Q_{k+1,m_s}$. In fact, $m_s$ is either $0$ or the largest
  positive number $m$ satisfying the following conditions:
  \begin{enumerate}[label=\alph*)]
  \item $i_{j-1}=i_{j+1}$ for $i\le j<i+m$;
  \item $i_{i+m-2}=i$;
  \item $i_{i+m-1}\ne i$.
  \end{enumerate}

  The simplex $\ts$ is non-degenerate: $i_{i+m_s-1}$ does not equal
  $i$ by hypothesis, while $i_{i+m_s}$ does not equal $i$ by the
  maximality of $m_s$. It is easily seen that $\ts$ has type
  $Q_{k+1,m_s}$.

  We see that $s=\p_{i+m_s}\ts$ is an inner face of $\ts$. The faces
  $\p_j\ts$, $j<i$, are either degenerate, lie in
  $\LAMBDA^n_i\cup\Delta^n$, or lie in $Q_{k,m_s-1}$. The faces
  $\p_j\ts$, $j>i$, are either degenerate, lie in
  $\LAMBDA^n_i\cup\Delta^n$, or lie in the boundary of simplex in
  $Q_{k+1,m}$, $m>m_s$.
\end{proof}

\begin{corollary}
  \label{lambda1}
  If $S\into T$ is an inner $n$-expansion of finite simplicial sets,
  then
  \begin{equation*}
    S\times_\D\DELTA \cup T \into T\times_\D\DELTA
  \end{equation*}
  is an inner $n$-expansion.
\end{corollary}
\begin{proof}
  The proof is by induction on the length of the filtration
  \begin{equation*}
    S = F_{-1}T \subset F_0T \subset F_1T \subset \dots \subset T
  \end{equation*}
  exhibiting $S\into T$ as an inner $n$-expansion. We see that there
  is a pushout square
  \begin{equation*}
    \begin{xy}
      \SQUARE<900,400>[\LAMBDA^{n_\ell}\cup\D^{n_\ell}`%
      \bigl(F_{\ell-1}T\times_\D\DELTA\bigr)\cup T`\DELTA^{n_\ell}`%
      \bigl(F_\ell T\times_\D\DELTA\bigr)\cup T;%
      ``(y_\ell\times_\D\DELTA)\cup x_\ell`]
    \end{xy}
  \end{equation*}
  It follows that
  $\bigl(F_\ell T\times_\D\DELTA\bigr)\cup T \into
  \bigl(F_{\ell-1}T\times_\D\DELTA\bigr)\cup T$
  is an $n_\ell$-expansion, where $n_\ell\ge n$. Since the inner
  $n$-expansions are closed under composition, the result follows.
\end{proof}

\begin{corollary}
  \label{lambda2}
  If $S\into T$ is an $m$-expansion of finite simplicial sets, where
  $m>1$, then
  \begin{equation*}
    S\times_\D\DELTA \into T\times_\D\DELTA
  \end{equation*}
  is an inner $m$-expansion.
\end{corollary}
\begin{proof}
  The proof is by induction on the number of nondegenerate simplices
  in $T\setminus S$. For the induction step, it suffices to prove that
  if $n>1$ and $0\le i\le n$, the inclusion
  $\LAMBDA^n_i \into \DELTA^n$ is an inner $n$-expansion.

  The action of the symmetric group $S_{n+1}$ on the simplicial set
  $\DELTA^n$ induces a transitive permutation of the subcomplexes
  $\LAMBDA^n_i$. Thus, it suffices to establish the result when
  $i=1$. But in this case, the inclusion
  $\LAMBDA^n_1\into\LAMBDA^n_1\cup\D^n$ is an inner $n$-expansion, and
  the result follows from Lemma~\ref{lambda}.
\end{proof}

We will also need some results involving the simplicial set
$\DELTA^1$. This simplicial set has two nondegenerate simplices of
dimension $k$, which we denote by
\begin{align*}
  \cell{k} &= (0,1,\dots) & \cell{k}^* &= (1,0,\dots) .
\end{align*}
Let $\cell{k}^\circ$ be the mirror of $\cell{k}$:
\begin{equation*}
  \cell{k}^\circ = (\dots,1,0) =
  \begin{cases}
    \cell{k} & \text{$k$ even} \\
    \cell{k}^* & \text{$k$ odd}
  \end{cases} .
\end{equation*}
In particular, the simplicial subset $\LAMBDA^1_1\into\DELTA^1$ may be
identified with the vertex $\cell{0}=(0)$.
\begin{lemma}
  \label{DD1}
  The inclusion
  \begin{equation*}
    \p\D^n\times\DELTA^1 \cup \D^n\times\LAMBDA^1_1 \into
    \D^n\times\DELTA^1
  \end{equation*}
  is an expansion, and an inner expansion if $n>0$.
\end{lemma}
\begin{proof}
  The expansion $\LAMBDA^1_1=\cell{0}\into\DELTA^1$ is obtained by
  successively adjoining the simplices $\cell{1}$, $\cell{2}$, \dots.

  The product $\D^n\times\DELTA^1$ is isomorphic to the iterated join
  of $n+1$ copies of $\DELTA^1$. Indeed, a $k$-simplex of
  $\D^n\times\DELTA^1$ may be identified with a pair consisting of a
  $k$-simplex $0^{a_0}\dots n^{a_n}$ of $\D^n$, where
  $a_0+\dots+a_n=k+1$, and a $k$-simplex $(i_0,\dots,i_k)$ of
  $\DELTA^1$. We may think of this $k$-simplex as a sequence of
  simplices $(\sigma_0,\dots,\sigma_n)$, where $\sigma_i$ is an
  $(a_i-1)$-simplex of $\DELTA^1$ if $a_i>0$, and is absent if
  $a_i=0$. Such a simplex is degenerate precisely when one of the
  $\sigma_i$ is degenerate. Denote the simplex
  $(i_0,\dots,i_k)\times0^{a_0}\dots n^{a_n}$ by
  $[\sigma_0;\dots;\sigma_n]$.

  The simplicial subset
  $\p\D^n\times\DELTA^1 \cup \D^n\times \LAMBDA^1_1 \subset \D^n\times
  \DELTA^1$
  is the union of the simplex $[\cell{0};\dots;\cell{0}]$, the
  simplices
  $[\sigma_0;\dots;\sigma_{i-1};;\sigma_{i+1};\dots;\sigma_n]$, and
  their faces.

  Let $S_{k,\ell,m}$ be the set of $k$-simplices in
  $\D^n\times\DELTA^1$ of the form
  \begin{equation*}
    [\cell{0};\dots;\cell{0};\cell{m};\sigma_{n-\ell+1};\dots;\sigma_n] ,
  \end{equation*}
  if $\ell<n$, and of the form
  \begin{equation*}
    [\cell{m}^\circ;\sigma_1;\dots;\sigma_n]
  \end{equation*}
  if $\ell=n$. The successive expansions of
  $\p\D^n\times\DELTA^1 \cup \D^n\times\LAMBDA^1_1$ along the
  simplices of $S_{k,\ell,m}$, in order first of ascending $k$, next
  of ascending $\ell$ (between $0$ and $n$), and lastly of ascending
  $m$ (between $1$ and $k-n$), exhibit the inclusion
  \begin{equation*}
    \p\D^n\times\DELTA^1 \cup \D^n\times\LAMBDA^1_1  \into
    \D^n\times\DELTA^1
  \end{equation*}
  as an inner expansion.
\end{proof}

\begin{corollary}
  A $k$-groupoid is a $k$-category.
\end{corollary}
\begin{proof}
  This follows from Lemma~\ref{expansion} and the special case of the
  lemma where $n=1$.
\end{proof}

\begin{corollary}
  \label{TD1}
  If $S\subset T$ is a simplicial subset containing the vertices of
  $T$, then the inclusion
  \begin{equation*}
    S\times\DELTA ^1 \cup T\times\LAMBDA^1_1 \into T\times\DELTA^1
  \end{equation*}
  is an inner expansion.
\end{corollary}

The following definition is modeled on Joyal's definition of
quasi-fibrations between quasi-categories \cite{Joyal}.
\begin{definition}
  A \textbf{quasi-fibration} $f:X\to Y$ of $k$-categories is a
  morphism of the underlying simplicial spaces such that
  \begin{enumerate}[label=\arabic*)]
  \item if $0<i<n$, the morphism
    \begin{equation*}
      X_n \too \Map(\L^n_i\into\D^n,f)
    \end{equation*}
    is a cover;
  \item if $i\in\{0,1\}$, the morphism
    \begin{equation*}
      \Map(\DELTA^1,X) \too \Map(\D^0\into\DELTA^1,f) = X_0
      \times_{Y_0} \Map(\DELTA^1,Y)
    \end{equation*}
    is a cover.
  \end{enumerate}
\end{definition}
Clearly, the morphism from a $k$-category $X$ to the terminal
simplicial space $e$ is a quasi-fibration.

The proof of the following lemma is the same as that of
Lemma~\ref{expansion}. Note that $\Map(S\into T,f)$ is isomorphic to
$\Map(\sk_{k+1}S\into\sk_{k+1}T,f)$ by Lemma~\ref{coskeletal}; this is
important, since $\Map(S\into T,f)$ is only defined \emph{a priori}
when $T$ is a finite simplicial set.
\begin{lemma}
  \label{inner-expansion}
  Let $T$ be a simplicial set such that $\sk_nT$ is finite for all
  $n$.
  \begin{enumerate}[label=\roman*)]
  \item Let $i:S\into T$ be an inner expansion, and let $f:X\to Y$ be
    a quasi-fibration of $k$-categories. Then the morphism
    \begin{equation*}
      \Map(T,X) \too \Map(S\into T,f)
    \end{equation*}
    is a cover.
  \item Let $i:S\into T$ be an inclusion, and let $f:X\to Y$ be a
    hypercover of $k$-categories. Then the morphism
    \begin{equation*}
      \Map(T,X) \too \Map(S\into T,f)
    \end{equation*}
    is a cover.
  \end{enumerate}
\end{lemma}  

We now introduce a functor $X\mapsto\GG(X)$ from $k$-categories to
$k$-groupoids, which may be interpreted as the $k$-groupoid of
quasi-invertible morphisms in $X$.
\begin{theorem} \mbox{}
  \begin{enumerate}[label=\roman*)]
  \item If $X$ is a $k$-category, then the simplicial space 
    \begin{equation*}
      \GG(X)_n = \Map(\DELTA^n,X)
    \end{equation*}
    is a $k$-groupoid.
  \item If $f:X\to Y$ is a quasi-fibration of $k$-categories, then
    \begin{equation*}
      \GG(f) : \GG(X) \to \GG(Y)
    \end{equation*}
    is a fibration of $k$-groupoids.
  \item If $f:X\to Y$ is a hypercover of $k$-categories, then
    \begin{equation*}
      \GG(f) : \GG(X) \to \GG(Y)
    \end{equation*}
    is a hypercover of $k$-groupoids.
  \end{enumerate}
\end{theorem}
\begin{proof}
  To prove Part i), we must show that the morphism
  \begin{equation*}
    \GG(X)_n\too\Map(\L^n_i,\GG(X)) ,
  \end{equation*}
  or equivalently, the morphism
  \begin{equation*}
    \Map(\DELTA^n,X) \too \Map(\LAMBDA^n_i,X) ,
  \end{equation*}
  is a cover for all $n>0$, and for $0\le i\le n$, and an isomorphism
  for $n>k$. For $n=1$, this is part of the definition of a
  quasi-fibration, and for $n>1$, it is a consequence of
  Corollary~\ref{lambda2}.

  The proof of Part~ii) is similar, since if $f:X\to Y$ is a
  quasi-fibration of $k$-categories, then the morphism
  \begin{equation*}
    \Map(\DELTA^n,X) \too \Map(\LAMBDA^n_i\into\DELTA^n,f) ,
  \end{equation*}
  is a cover for all $n>0$, and for $0\le i\le n$, by the same
  argument.

  To prove Part~iii), we must show that if $f:X\to Y$ is a hypercover,
  the morphism
  \begin{equation*}
    \GG(X)_n \too \Map(\p\D^n\into\D^n,\GG(f)) ,
  \end{equation*}
  or equivalently, the morphism
  \begin{equation*}
    \Map(\DELTA^n,X) \too \Map(\p\DELTA^n\into\DELTA^n,f) ,
  \end{equation*}
  is a cover for all $n\ge0$: this follows from
  Lemma~\ref{cofibration}, applied to the inclusion of simplicial sets
  $\p\DELTA^n\into\DELTA^n$.
\end{proof}

It is clear that $\GG$ takes pullbacks to pullbacks. We will show that
$k$-categories form a category of fibrant objects, and that $\GG$ is
an exact functor from this category to the category of $k$-groupoids.

The main step which remains in the proof that $k$-categories form a
category of fibrant objects is the construction of a simplicial
resolution for $k$-categories. We use the following refinement of
Lemma~\ref{Moore2}, which was already implicit in the proof of
Lemma~\ref{Moore}.
\begin{lemma}
  \label{Moore3}
  Let $T$ be a finite simplicial set, and let $S\into T$ be a
  simplicial subset. Then the morphism
  \begin{equation*}
    \D^m\times S \cup \L^m_i\times T \into \D^m\times T , \quad
    0<i<m ,
  \end{equation*}
  is an inner $m$-expansion, and the morphism
  \begin{equation*}
    S \times \D^n \cup T \times \L^n_j \into T \times \D^n , \quad
    0<j<n ,
  \end{equation*}
  is an inner $n$-expansion.
\end{lemma}

\begin{definition}
  Define $\P_nX$ to be the simplicial space
  \begin{equation*}
    (\P_nX)_m = \Map(\D^m\times\DELTA^n,X) .
  \end{equation*}
\end{definition}

\begin{theorem}
  The functor $\P_\bull X$ is a simplicial resolution.
\end{theorem}
\begin{proof}
  Let $f:X\to Y$ be a quasi-fibration. By Lemma~\ref{Moore3}, the
  inclusion
  \begin{equation*}
    \L^m_i\times\DELTA^n\cup\D^m\times\p\DELTA^n \into
    \D^m\times\DELTA^n
  \end{equation*}
  is an inner expansion for $0<i<m$. Applying
  Lemma~\ref{inner-expansion}, we conclude that the morphism
  \begin{equation*}
    \Map(\D^m\times\DELTA^n,X) \too
    \Map(\L^m_i\times\DELTA^n\cup\D^m\times\p\DELTA^n\into
    \D^m\times\DELTA^n,f)
  \end{equation*}
  is a cover.

  By Corollary~\ref{TD1}, the inclusion
  \begin{equation*}
    \DELTA^1\times\p\DELTA^n\cup\LAMBDA^1_1\times\DELTA^n
    \into \DELTA^1\times\DELTA^n
  \end{equation*}
  is an inner expansion for $n>0$. It follows by
  Lemma~\ref{inner-expansion} that the morphism
  \begin{equation*}
    \Map(\DELTA^1\times\DELTA^n,X) \too
    \Map(\DELTA^1\times\p\DELTA^n\cup\LAMBDA^1_1\times\DELTA^n\into
    \DELTA^1\times\DELTA^n,f)
  \end{equation*}
  is a cover for $n>0$. Together, these two results show that the
  simplicial morphism
  \begin{equation*}
    \P_nX \too \P_{\p\D^n}X\times_{P_{\p\D^n}Y}\P_nY
  \end{equation*}
  is a quasi-fibration for $n>0$.

  By Corollary~\ref{lambda2} and Lemma~\ref{Moore3}, the inclusion
  \begin{equation*}
    \p\D^m\times\DELTA^n\cup\D^m\times\LAMBDA^n_j \into
    \D^m\times\DELTA^n
  \end{equation*}
  is an inner expansion for $n>1$ and $0\le j\le n$. It follows that
  the morphism
  \begin{equation*}
    \Map(\D^m\times\DELTA^n,X) \too
    \Map(\p\D^m\times\DELTA^n\cup\D^m\times\LAMBDA^n_j\into
    \D^m\times\DELTA^n,f)
  \end{equation*}
  is a cover, and hence that the simplicial morphism
  \begin{equation*}
    \P_nX \too \P_{\L^n_i}X \times_{\P_{\L^n_i}Y} \P_nY
  \end{equation*}
  is a hypercover for $n>1$.

  Let $f:X\to Y$ be a hypercover. Applying Lemma~\ref{cofibration}, we
  see that the morphism
  \begin{equation*}
    \Map(\D^m\times\DELTA^n,X) \too
    \Map(\p\D^m\times\DELTA^n\cup\D^m\times\p\DELTA^n\into
    \D^m\times\DELTA^n,f)
  \end{equation*}
  is a cover for $n>0$, and hence the simplicial morphism
  \begin{equation*}
    \P_nX \too \P_{\p\D^n}X\times_{P_{\p\D^n}Y}\P_nY
  \end{equation*}
  is a hypercover for $n>0$.
\end{proof}

The following lemma is the analogue of Lemma~\ref{Hard} for
$k$-categories.
\begin{lemma}
  If $f:X\to Y$ is a fibration of $k$-categories, and $g:Y\to Z$ and
  $g\*f$ are hypercovers, then $f$ is a hypercover.
\end{lemma}
\begin{proof}
  The proof of Lemma~\ref{Hard} extends to this setting as
  well. Indeed, the proof contained there establishes that the
  morphism $X_n\to\Map(\p\D^n\into\D^n,f)$ is a cover for $n>0$. It
  remains to show that $f_0:X_0\to Y_0$ is a cover, which follows from
  Lemma~\ref{Hard} applied to the morphisms $\GG(f)$ and $\GG(g)$.
\end{proof}

With these results in hand, we may easily adapt the proof of
Theorem~\ref{main} to prove the following result.
\begin{theorem}
  The category of $k$-categories is a category of fibrant objects.
\end{theorem}

The following corollary is immediately implied by Lemma~\ref{P(f)}
(``Brown's Lemma'').
\begin{corollary}
  If $f:X\to Y$ is a weak equivalence of $k$-categories, then
  \begin{equation*}
    \GG(f):\GG(X)\too\GG(Y)
  \end{equation*}
  is a weak equivalence of $k$-groupoids.
\end{corollary}

We have the following analogue of Theorem~\ref{We}.
\begin{theorem}
  A morphism $f:X\to Y$ of $k$-categories is a weak equivalence if and
  only if the morphism
  \begin{equation*}
    X_0 \times_{Y_0} \Map(\DELTA^1,Y) \too Y_0
  \end{equation*}
  is a cover, and the morphisms
  \begin{equation*}
    \Map(\D^n\into\DELTA^1\star\D^{n-1},f) \too
    \Map(\p\D^n\into
    \DELTA^1\star\p\D^{n-1}\cup\LAMBDA^1_0\star\D^{n-1},f)
  \end{equation*}
  are covers for $n\ge0$.
\end{theorem}
\begin{proof}
  The morphism $f$ is a weak equivalence if and only if the morphisms
  \begin{equation}
    \label{we**}
    \Map(\D^n\into\D^n\times\DELTA^1,f) \too
    \Map(\p\D^n\times\LAMBDA^1_1 \into
    \p\D^n\times\DELTA^1\cup\D^n\times\LAMBDA^1_1,f)
  \end{equation}
  are covers for all $n\ge0$. For $n=0$, this is the first hypothesis
  of the theorem. Thus, from now on, we take $n>0$.

  We have seen in Lemma~\ref{DD1} that the simplicial set
  $\D^n\times\DELTA^1$ is an inner expansion of
  $\p\D^n\times\DELTA^1\cup\D^n\times\LAMBDA^1_1$, by the successive
  adjunction of the simplices
  $[\cell{0};\dots;\cell{0};\cell{m};\sigma_{n-\ell+1};\dots;\sigma_n]$
  and
  \begin{equation*}
    [\cell{m}^\circ;\sigma_1;\dots;\sigma_n] .
  \end{equation*}
  Of these simplices, only one, namely
  $[\cell{1}^*;\cell{0}^*;\dots;\cell{0}^*]\in S_{n+1,n,1}$, has a
  face in the simplicial subset
  $\D^n\times\LAMBDA^1_1\subset\D^n\times\DELTA^1$. Thus, the morphism
  \eqref{we**} factors into a sequence of horn-filler morphisms
  indexed by this sequence of simplices, all of which are seen to be
  covers, except possibly the one corresponding to the simplex
  $[\cell{1}^*;\cell{0}^*;\dots;\cell{0}^*]$. But the morphism
  corresponding to this simplex is a cover under the hypotheses of the
  theorem.

  Now suppose that \eqref{we**} is a cover for $n>0$. The map
  \begin{equation*}
    0^{a_0}\dots n^{a_k} \times i_0\dots i_k \mapsto
    0^{a_0}\dots n^{a_k} \times i_0\dots i_{a_0-1}0\dots0
  \end{equation*}
  from $\D^n\times\DELTA^1$ to $\DELTA^1\star\D^{n-1}$ takes
  $\p\D^n\times\DELTA^1\cup\D^n\times\LAMBDA^1_1$ to
  $\DELTA^1\star\p\D^{n-1}\cup\LAMBDA^1_0\star\D^{n-1}$ and induces a
  pullback square
  \begin{equation*}
    \small
    \begin{xy}
      \SQUARE<1900,500>[\Map(\D^n \into \DELTA^1\star\D^{n-1},f)`
      \Map(\D^n\into\D^n\times\DELTA^1,f)`
      \Map(\p\D^n\into
      \DELTA^1\star\p\D^{n-1}\cup\LAMBDA^1_0\star\D^{n-1},f)`
      \Map(\p\D^n\into\p\D^n\times\DELTA^1\cup\D^n\times\LAMBDA^1_1,f);```]
    \end{xy}
  \end{equation*}
  This completes the proof of the theorem.
\end{proof}

\section{Regular $k$-categories}

If $\CV$ is a regular descent category, it is natural to single out
the following class of $k$-categories.
\begin{definition}
  A \textbf{regular} $k$-category is a $k$-category $X$ such
  that the morphism
  \begin{equation*}
    \Map(\DELTA^1,X) \too \Map(\D^1,X) \cong X_1
  \end{equation*}
  induced by the inclusion $\D^1\into\DELTA^1$ is regular.
\end{definition}

Since $\D^1\into\DELTA^1$ is an expansion, every $k$-groupoid is a
regular $k$-category.

\begin{proposition}
  \label{G}
  If $X$ is a regular $k$-category, then for all $n\ge0$, the morphism
  \begin{equation*}
    \Map(\DELTA^n,X) \too \Map(\D^n,X) \cong X_n
  \end{equation*}
  induced by the inclusion $\D^n\into\DELTA^n$ is regular.
\end{proposition}
\begin{proof}
  Let $\T^n_i\subset\D^n$ be the union of the $1$-simplices
  \begin{equation*}
    (j-1,j) , \quad 1\le j\le i .
  \end{equation*}
  For $k>0$, let $Q_k$ be the set of $k$-simplices of $\D^n$ such that
  $i_1=i_0+1$. In particular, $Q_1$ is the set of $1$-simplices in
  $\T^n_n$.

  Let $k>1$. Given a simplex $(i_0,\dots,i_k)\in Q_k$, the faces
  $\p_j(i_0,\dots,i_k)$ lie in $Q_{k-1}$ for $j>1$, while
  $\p_0(i_0,\dots,i_k)$ either lies in $Q_{k-1}$, if $i_2=i_1+1$, or
  equals $\p_1(i_1,i_1+1,i_2,\dots,i_k)$ if $i_2>i_1+1$.

  On the other hand, $\p_1(i_0,\dots,i_k)$ lies neither in $Q_{k-1}$
  nor is it a face of any simplex $(i_0',\dots,i_k')\in Q_k$ with
  $i_0'+\dots+i_k'>i_0+\dots+i_k$. This shows that the inclusion
  $\T^n_n\into\D^n$ is an inner expansion, in which the simplices of
  $Q_k$ are attached in order of increasing $k\ge2$, and for fixed
  $k$, in order of decreasing $i_0+\dots+i_k$.

  Let $\TT^n_i = (\T^n_i\o_\D\DELTA) \cup \D^n \subset \DELTA^n$. By
  Lemma~\ref{lambda1}, $\TT^n_n\into\DELTA^n$ is an inner expansion.
  Hence the morphism
  \begin{equation*}
    \Map(\DELTA^n,X) \too \Map(\TT^n_n,X)
  \end{equation*}
  is a cover, and hence regular. For each $1\le i\le n$, the morphism
  \begin{equation*}
    \Map(\TT^n_i\D^n,X) \too \Map(\TT^n_{i-1}\D^n,X)
  \end{equation*}
  is regular, since it may be realized as the pullback of a regular
  morphism:
  \begin{equation*}
    \begin{xy}
      \Square[\Map(\TT^n_i,X)`\Map(\DELTA^1,X)`
      \Map(\TT^n_{i-1},X)`\Map(\D^1,X);```]
    \end{xy}
  \end{equation*}
  This completes the proof of the theorem, since $\TT^n_0=\D^n$, and
  the composition of regular morphisms is regular.
\end{proof}

Let $\G(X)_n$ be the image of the regular morphism $\GG(X)_n\to X_n$.
The spaces $\G(X)_n$ form a simplicial space, and for each $n$, the
morphism $\GG(X)_n\to\G(X)_n$ (coimage of $\GG(X)_n\to X_n$) is a
cover. We call $\G(X)_1$ the space of \textbf{quasi-invertible}
morphisms.

It follows from the proof of Theorem~\ref{G} that $\G(X)_n$ is
the image of the morphism
\begin{equation*}
  \Map(\T^n_n,\GG(X)) \times_{\Map(\T^n_n,X)} X_n \too X_n .
\end{equation*}

\begin{lemma}
  \label{GGG}
  $\GG(\GG(X)) \cong \GG(\G(X)) \cong \GG(X)$
\end{lemma}
\begin{proof}
  In order to prove that $\GG(\GG(X))$ is isomorphic to $\GG(X)$, it
  suffices to show that for all $k,n\ge0$,
  \begin{equation*}
    \Map(\D^k,\DELTA^n)\cong\Map(\DELTA^k,\DELTA^n) .
  \end{equation*}
  Since $\DELTA^k$ is the nerve of the groupoid $\[k\]$,
  we see that $\Map(\DELTA^k,\DELTA^n)$ may be identified with the set
  of functors from $\[k\]$ to $\[n\]$. But a functor from $\[k\]$ to $\[n\]$
  determines, and is determined by, a functor from $[k]$ to $\[n\]$,
  i.e.\ by a $k$-simplex of the nerve $\DELTA^n=N_\bull\[n\]$ of
  $\[n\]$.

  Applying the functor $\GG_n$ to the composition of morphisms
  \begin{equation*}
    \GG(X) \to \G(X) \to X ,
  \end{equation*}
  we obtain a factorization of the identity map of $\GG(X)_n$:
  \begin{equation*}
    \GG(\GG(X))_n \cong \GG(X)_n \to \GG(\G(X))_n \to \GG(X)_n .
  \end{equation*}
  Since the functor $\GG_n$ is a limit, it preserves
  monomorphisms. Thus the morphism from $\GG(\G(X))_n$ to $\GG(X)_n$
  is a monomorphism, and since it has a section, an isomorphism.
\end{proof}

The statement and proof of the following lemma are similar to those of
Lemma~\ref{lambda}.
\begin{lemma}
  \label{mu}
  The inclusion $\p\DELTA^n\cup\D^n\into\DELTA^n$ is an expansion.
\end{lemma}
\begin{proof}
  Let $Q_{k,m}$, $0\le m<n$ be the set of non-degenerate $k$-simplices
  $s=(i_0\dots i_k)$ of $\DELTA^n$ which satisfy the following
  conditions:
  \begin{enumerate}[label=\alph*)]
  \item $s$ is not contained in $\p\LAMBDA^n\cup\D^n$;
  \item $i_j=j$ for $i\le j\le m$;
  \item $\{i_{m+1},\dots,i_n\}=\{m,\dots,n\}$.
  \end{enumerate}
  Let $Q_k$ be the union of the sets $Q_{k,m}$.

  The simplicial set $\DELTA^n$ is obtained from $\LAMBDA^n_i\cup\D^n$
  by inner expansions along the simplices of type $Q_{k,m}$ in order
  first of increasing $k$, then of decreasing $m$. (The order in which
  the simplices are adjoined within the sets $Q_{k,m}$ is
  unimportant.)

  Given a non-degenerate simplex $s=(i_0,\dots,i_k)$ which does not
  lie in the union of $\p\DELTA^n\cup\Delta^n$ and $Q_k$, let $m$ be
  the largest integer such that $i_j=j$ for $j<m$. Thus
  \begin{equation*}
    s = (0,\dots,m-1,i_m,\dots,i_k) ,
  \end{equation*}
  and $i_m\ne m$. The infimum $\ell$ of the set $\{i_m,\dots,i_k\}$
  equals $m$: it cannot be any larger, or the simplex would lie in
  $\p\DELTA^n$, and it cannot be any smaller, or the simplex would lie
  in $Q_k$. Define the simplex
  \begin{equation*}
    \ts = (0,\dots,m,i_m,\dots,i_k)
  \end{equation*}
  in $Q_{k+1,m}$. We have $s=\p_m\ts$.

  If $m$ occurs more than once in the sequence $\{i_m,\dots,i_k\}$,
  then the remaining faces of the simplex $\ts$ are either degenerate,
  or lie in the union of $\p\DELTA^n\cup\Delta^n$ and $Q_k$. If $m$
  occurs just once in this sequence, say $i_\ell=m$, then all faces of
  the simplex $\ts$ other than $s=\p_m\ts$ and $\p_{\ell+1}\ts$ are
  either degenerate, or lie in the union of $\p\DELTA^n\cup\Delta^n$
  and $Q_k$, while $\p_{\ell+1}\ts$ is a face of a simplex of type
  $Q_{k+1,m'}$, where $m'>m$.
\end{proof}

This lemma implies that the natural morphism $\GG(X)\to X$ is a
hypercover when $X$ is a $k$-groupoid, even if the descent category is
not assumed to be regular.

The following theorem is related to results of Rezk \cite{Rezk} and
Joyal and Tierney~\cite{JT}.
\begin{theorem}
  \label{Regular}
  If $X$ is a regular $k$-category, then $\G(X)$ is a $k$-groupoid,
  and the induced morphism
  \begin{equation*}
    \GG(X) \too \G(X)
  \end{equation*}
  is a hypercover.
\end{theorem}
\begin{proof}
  For $n>0$, consider the assertions \\
  $\textup{A}_n$: for all $0\le i\le n$, the morphism
  $\G(X)_n\to\Map(\L^n_i,\G(X))$ is a cover; and \\
  $\textup{B}_n$: for all $0\le i\le n$, the morphism
  \begin{equation*}
    \GG(X)_n \too \Map(\L^n_i\to\D^n,\GG(X)\to\G(X))
  \end{equation*}
  is a cover. These imply that $\G(X)$ is a $k$-groupoid.

  Let us demonstrate $\textup{A}_1$. In the commuting diagram
  \begin{equation*}
    \begin{xy}
      \Atriangle/{>}`{>}`{.>}/<700,400>[\GG(X)_1`\G(X)_1`\GG(X)_0\cong
      X_0;``]      
    \end{xy}
  \end{equation*}
  the solid arrows are covers, hence by Axiom~\ref{cancellation}, the
  bottom arrow is a cover.

  Consider the commuting diagram
  \begin{equation*}
    \begin{xy}
      \Atriangle/{.>}`{>}`{.>}/<1000,500>%
      [\GG(X)_n\times_{\Map(\L^n_i,\G(X))}\G(X)_n`\GG(X)_n`%
      \Map(\L^n_i\to\D^n,\GG(X)\to\G(X));``]
    \end{xy}
  \end{equation*}
  in which the solid arrow is a cover. If $\textup{A}_n$ holds, the
  left-hand arrow is a cover, and hence by Axiom~\ref{cancellation},
  so is the bottom arrow, establishing $\textup{B}_n$.

  Suppose that $T$ is a finite simplicial set and $S\into T$ is an
  expansion obtained by attaching simplices of dimension at most $n-1$
  to $S$. Suppose that $\textup{B}_{n-1}$ holds. Then the same proof
  as for Lemma~\ref{expansion} shows that the morphism
  \begin{equation*}
    \Map(T,\GG(X)) \to \Map(S\into T,\GG(X)\to\G(X))
  \end{equation*}
  is a cover. Applying this argument to the expansion
  $\D^0\into\L^n_i$ shows that
  \begin{equation*}
    \Map(\L^n_i,\GG(X)) \too \Map(\L^n_i,\G(X))
  \end{equation*}
  is a cover. In the commuting diagram
  \begin{equation*}
    \begin{xy}
      \Square/{>}`{>}`{>}`{.>}/%
      [\GG(X)_n`\Map(\L^n_i,\GG(X))`\G(X)_n`\Map(\L^n_i,\G(X));```]      
    \end{xy}
  \end{equation*}
  the solid arrows are covers, hence by Axiom~\ref{cancellation}, so
  is the bottom arrow, establishing $\textup{A}_n$.

  Now that we know that $\G(X)$ is a $k$-groupoid, it follows from
  Lemma~\ref{mu} that $\GG(X)\to\G(X)$ is a hypercover.
\end{proof}

\section{The nerve of a differential graded algebra}

In this final section, we give an application of the formalism
developed in this paper to the study of the nerve of a differential
graded algebra $A$ over a field $\K$. There are different variants of
this construction: we give the simplest, in which the differential
graded algebra $A$ is finite-dimensional in each dimension and
concentrated in degrees $>-k$. Working in the descent category of
schemes of finite type, with surjective smooth morphisms (respectively
smooth morphisms) as covers (respectively regular morphisms), we will
show that the nerve of $A$ is a regular $k$-category.

In the special case that $A=M_N(\K)$ is the algebra of $N\times N$
square matrices, our construction produces the nerve of the monoid
$\End(\K^N)$: the associated $1$-groupoid $\G(N_\bull A)$ is the nerve
of the algebraic group $\GL(N)$. If $V$ is a perfect complex of
amplitude $k$, then $\GG(N_\bull\End(V))$ is the $k$-groupoid of
quasi-automorphisms of $V$. A straightforward generalization of this
construction from differential graded algebras to differential graded
categories yields the stack of perfect complexes: in a sequel to this
paper, we show how this gives a new construction of the derived stack
of perfect complexes of To\"en and Vezzosi \cite{tv}.

Let $A$ be a differential graded algebra over a field $\K$, with
differential $d:A^\bull\to A^{\bull+1}$. The curvature map is the
quadratic polynomial
\begin{equation*}
  \Phi(\mu) = d\mu + \mu^2 : A^1\to A^2 .   
\end{equation*}
The Maurer-Cartan locus $\MC(A)=V(\Phi)\subset A^1$ is the zero locus
of $\Phi$.

The graded commutator of elements $a\in A^i$ and $b\in A^j$ is defined
by the formula
\begin{equation*}
  [a,b] = ab - (-1)^{ij} ba \in A^{i+j} .
\end{equation*}
In particular, if $\mu\in A^1$, then
\begin{equation*}
  [\mu,a] = \mu a - (-1)^i a\mu \in A^{i+1} .
\end{equation*}
If $\mu$ lies in the Maurer-Cartan locus, the operator
$d_\mu:a\mapsto da+[\mu,a]$ is a differential.

Given $\mu$ and $\nu$ lying in the Maurer-Cartan locus of $A^\bull$,
define a differential $d_{\mu,\nu}$ on the graded vector space
underlying $A$ by the formula
\begin{equation*}
  A^i \ni a \mapsto d_{\mu,\nu}a = d a + \mu a - (-1)^i a\nu \in
  A^{i+1} .
\end{equation*}

Let $C^\bull(\D^n)$ be the differential graded algebra of normalized
simplicial cochains on the $n$-simplex $\D^n$ (with coefficients in
the field $\K$): this algebra is finite-dimensional, of dimension
$\binom{n+1}{i+1}$ in degree $i$. An element
$a\in C^\bull(\D^n) \o A^\bull$ corresponds to a collection of elements
\begin{equation*}
  ( a_{i_0\dots i_k} \in A^{i-k} \mid 0\le i_0<\dots<i_k\le n ) ,
\end{equation*}
where $a_{i_0\dots i_k}$ is the evaluation of the cochain $a$ on the
face of the simplex $\D^n$ with vertices $\{i_0,\dots,i_k\}$.

The differential on the differential graded algebra
$C^\bull(\D^n)\o A$ is the sum of the simplicial differential on
$C^\bull(\D^n)\o A$ and the internal differential of $A$:
\begin{equation*}
  (\delta a)_{i_0\dots i_k} = \sum_{\ell=0}^k (-1)^\ell
  a_{i_0\dots\widehat{\imath}{}_{i_\ell}\dots i_k} + (-1)^k d(a_{i_0\dots
    i_k}) .
\end{equation*}
The product of $C^\bull(\D^n)\o A$ combines the Alexander-Whitney
product on simplicial cochains with the product on $A$: if $a$ has
total degree $j$, then
\begin{equation*}
  (a\cup b)_{i_0\dots i_k} = \sum_{\ell=0}^k (-1)^{(j-\ell)(k-\ell)}
  a_{i_0\dots i_\ell} b_{i_\ell\dots i_k} .
\end{equation*}

The \textbf{nerve} of a differential graded algebra $A$ is the
simplicial scheme $N_\bull A$ such that $N_nA$ is the Maurer-Cartan
locus of $C^\bull(\D^n)\o A$:
\begin{equation*}
  N_nA = \MC(C^\bull(\D^n)\o A) .
\end{equation*}
If $T$ is a finite simplicial set, the Yoneda lemma implies that the
scheme of morphisms from $T$ to $N_\bull A$ is the Maurer-Cartan set
of the differential graded algebra $C^\bull(T)\o A$.

A simplex $\bbmu\in N_nA$ consists of a collection of elements of $A$
\begin{equation*}
  \bbmu = \bigl\{ \mu_{i_0\dots i_k} \in A^{1-k} \mid 0\le
  i_0<\ldots<i_k\le n \bigr\} ,
\end{equation*}
such that the following Maurer-Cartan equations hold: for
\begin{equation*}
  0\le i_0<\ldots<i_k\le n ,
\end{equation*}
we have
\begin{equation*}
  (-1)^k \, (d\bbmu+\bbmu^2)_{i_0\dots i_k}
  = d\mu_{i_0\dots i_k} + \sum_{\ell=0}^k (-1)^{k-\ell} \,
  \mu_{i_0\dots\widehat{\imath}{}_\ell\dots i_k} + \sum_{\ell=0}^k
  (-1)^{k\ell} \, \mu_{i_0\dots i_\ell} \mu_{i_\ell\dots i_k} = 0 .
\end{equation*}

The components $\mu_i$ and $\mu_{ij}$ play a special role in the
Maurer-Cartan equation. The components $\mu_i$ are Maurer-Cartan
elements of $A$, and determine differentials $d_{ij}:A^\bull\to
A^{\bull+1}$ by the formula
\begin{equation*}
  d_{ij}a = da + \mu_ia - (-1)^{|a|} a\mu_j .
\end{equation*}
In terms of the translate $f_{ij}=1+\mu_{ij}$ of the coefficient
$\mu_{ij}$, the Maurer-Cartan equation for $\mu_{ij}$ becomes
\begin{equation*}
  d_{ij}f_{ij} = 0 .
\end{equation*}
The Maurer-Cartan equation for $\mu_{ijk}$ may be rewritten
\begin{equation*}
  d_{ik}\mu_{ijk} + f_{ij}f_{jk} - f_{ik} = 0 .
\end{equation*}
In other words, $\mu_{ijk}$ is a homotopy between $f_{ij}f_{jk}$ and
$f_{ik}$. For $n>2$, the Maurer-Cartan equation becomes
\begin{multline*}
  d_{i_0i_k}\mu_{i_0\dots i_k} + \sum_{\ell=1}^{k-1} (-1)^{k-\ell} \,
  \mu_{i_0\dots\widehat{\imath}{}_\ell\dots i_k} \\ + (-1)^k \,
  f_{i_0i_1} f_{i_1\dots i_k} + \mu_{i_0\dots i_{k-1}}
  \mu_{i_{k-1}i_k}
  + \sum_{\ell=2}^{k-2} (-1)^{k\ell} \, \mu_{i_0\dots i_\ell}
  \mu_{i_\ell\dots i_k} = 0 .
\end{multline*}

The following is the main result of this section.
\begin{theorem}
  \label{Smooth}
  Let $A$ be a differential graded algebra such that $A^i$ is
  finite-dimensional for $i\le1$, and vanishes for $i\le-k$. Then
  $N_\bull A$ is a regular $k$-category.
\end{theorem}
\begin{proof}

The proof divides into three parts.
\begin{enumerate}[label=\arabic*)]
\item If $0<i<n$, the morphism $N_nA\to\Map(\L^n_i,N_\bull A)$ is a
  smooth epimorphism, and an isomorphism if $n>k$.
\item The morphisms $\Map(\DELTA^1,N_\bull A)\to\MC(A)$ are smooth.
\item The morphism $\Map(\DELTA^1,N_\bull A)\to N_1A$ is smooth.
\end{enumerate}

Part 1) is established by rearranging the Maurer-Cartan equations for
$\mu_{0\dots n}$ and $\mu_{0\dots\widehat{\imath}\dots n}$ to give a
natural isomorphism $N_nA\cong\Map(\L^n_i,N_\bull A)\times A^{1-n}$:
\begin{align*}
  \mu_{0\dots n} &= x \in A^{1-n} \\
  \mu_{0\dots\widehat{\imath}\dots n}
                 &= - (-1)^{n-i} d_{0n}x - (-1)^i f_{01}\mu_{1\dots n}
                   - (-1)^{n-i} \mu_{0\dots n-1}f_{n-1,n} \\
                 & \quad - \sum_{\ell\notin\{0,i,n\}} (-1)^{\ell-i}
                   \mu_{0\dots\widehat{\ell}\dots n} -
                   \sum_{\ell=2}^{n-2} (-1)^{n\ell-n+i}
                   \mu_{0\dots\ell} \mu_{\ell\dots n} \in A^{2-n} .
\end{align*}
The case $n=2$ is slightly special:
\begin{align*}
  \mu_{012} &= x\in A^{-1} \\
  \mu_{02} &= dx + \mu_0x + x\mu_2 + f_{01}f_{12} - 1 \in A^0 .
\end{align*}

To establish Parts 2) and 3), we will use an alternative
representation of the algebra $C^\bull(\DELTA^1)\o A$ in terms
of $2\times2$ matrices with coefficients in $A[u]$, where $u$ is a
formal variable of degree $2$.

Associate to a differential graded algebra $A$ the auxilliary
differential graded algebra $\UA$, such that $\UA^k$ is the space of
$2\times2$ matrices
\begin{equation*}
  \UA^k = \left\{
    \begin{pmatrix} 
      \alpha_{00} & \alpha_{01} \\
      \alpha_{10} & \alpha_{11}
    \end{pmatrix} \, \middle| \, \alpha_{ij} \in A^{k+i-j}[u] \right\} .
\end{equation*}
Composition is the usual matrix product. Let $d:\UA\to\UA$ be the
differential given by the formula
\begin{equation*}
  (da)_{ij} = (-1)^i \, d\bigl(\alpha_{ij}\bigr) .
\end{equation*}
Let $\VA\subset\UA$ be the differential graded subalgebra 
\begin{equation*}
  \VA = \left\{
    \begin{pmatrix} 
      \alpha_{00} & \alpha_{01} \\
      \alpha_{10} & \alpha_{11}
    \end{pmatrix} \in \UA \, \middle| \, \alpha_{10}(0)=0 \right\} .
\end{equation*}
In other words, the bottom left entry $\alpha_{10}$ of the matrix has
vanishing constant term. Let $a_0\in\VA$ be the element
\begin{equation*}
  a_0 = \begin{pmatrix}
    0 & 1 \\
    u & 0
  \end{pmatrix} .
\end{equation*}
The following lemma is a straightforward calculation.
\begin{lemma}
  The map from $C^\bull(\DELTA^1)\o A$ to $\VA$ given by the formula
  \begin{equation*}
    x \mapsto
    \psi(x) =
    \begin{pmatrix}
      x_0 + ux_{010} + u^2x_{01010} + \dots
      & x_{01} + ux_{0101} + u^2x_{010101} + \dots \\[3pt]
      ux_{10} + u^2x_{1010} + \dots & -
      x_1 - ux_{101} - u^2x_{10101} - \dots
  \end{pmatrix}    
  \end{equation*}
  is an isomorphism of differential graded algebras between
  $C^\bull(\DELTA^1)\o A$ and $\VA$ with differential
  \begin{equation*}
    \delta x = dx + [a_0,x] .
  \end{equation*}
\end{lemma}
\begin{corollary}
  The morphism
  \begin{equation*}
    \bbmu \mapsto a(\bbmu) = a_0 + \psi(\bbmu)
  \end{equation*}
  induces an isomorphism of schemes between
  $\N_1A=\MC(C^\bull(\DELTA)\o A)$ and
  \begin{equation*}
    Z(da + a^2 - u1) \subset \VA^1 .
  \end{equation*}
\end{corollary}

A Maurer-Cartan element $\bbmu=(\mu_0,\mu_1,\mu_{01})$ is
quasi-invertible if
\begin{equation*}
  f = 1+\mu_{01}
\end{equation*}
is quasi-invertible in $A^0$: that is, there exist elements $g\in A^0$
and $h$ and $k\in A^{-1}$ such that
\begin{align*}
  dh+[\mu_0,h] &= fg-1 , & dk+[\mu_1,k] &= gf-1 .
\end{align*}
The following result (with a different proof) is due to Markl
\cite{Markl}.
\begin{proposition}
  Every quasi-invertible point of $N_1A$ may be lifted to a point of
  $\N_1A$.
\end{proposition}
\begin{proof}
  Consider the matrices
  \begin{align*}
    \alpha &=                    
    \begin{pmatrix}
     \mu_0 & f \\ 0 & - \mu_1
    \end{pmatrix} \in \VA^1 &
    \beta &=                    
    \begin{pmatrix}
      h & h(fk-hf) \\ g & -k + g(fk-hf)
    \end{pmatrix} \in \VA^{-1}
  \end{align*}
  It is easily checked that $d\beta+[\alpha,\beta]=1$. Let $C_n$ be
  the $n$th Catalan number. The matrix
  \begin{equation*}
    a = \alpha + u \sum_{n=0}^\infty (-u)^n C_n \, \beta^{2n+1} \in
    \VA^1
  \end{equation*}
  solves the equation $da+a^2=u1$, and corresponds to an element of
  $\N_1A$ lifting $\bbmu\in N_1A$. (The sum defining $a$ is finite,
  since the differential graded algebra $A^\bull$ is bounded below.)
\end{proof}

The following lemma is our main tool in the proofs of Parts~2) and 3).
\begin{lemma}
  \label{smooth}
  Let $A$ be a differential graded algebra such that $A^1$ is
  finite-dimensional. Let $h:A^\bull\to A^{\bull-1}$ be an operator on
  $A$ satisfying the following conditions:
  \begin{enumerate}[label=\alph*)]
  \item $hdh=h$ and $h^2=0$;
  \item the image of $p=dh+hd$ is an ideal $I\subset A$.
  \end{enumerate}
  Then the natural morphism $\MC(A)\to\MC(A/I)$ is smooth at
  $0\in\MC(A)$.
\end{lemma}
\begin{proof}
  Let $U$ be the open neighbourhood of $0$ in $A^1$ on which the
  determinant of the linear transformation
  \begin{equation*}
    1 + h\ad(\mu) : A^1 \too A^1
  \end{equation*}
  is nonzero. We will show that the projection $\MC(A)\to\MC(A/I)$ is
  smooth on the open subset $U\cap\MC(A)$.

  There is an isomorphism between $\MC(A)$ and the variety
  \begin{equation*}
    \V = Z(p\nu,(1-p)x,dhx-y,\Phi(\nu)+d_\nu x+x^2)
    \subset \X = \{ (\nu,x,y) \in A^1\times A^1\times A^1 \} ,
  \end{equation*}
  induced by the morphism taking $\mu\in A^1$ to
  $((1-p)\mu,p\mu,h\mu)$. Likewise, there is an isomorphism between
  $\MC(A/I)$ and the variety
  \begin{equation*}
    Z(p\nu,(1-p)\Phi(\nu)) \subset \{ \nu \in A^1 \} .
  \end{equation*}
  It follows that the variety
  \begin{equation*}
    \W = Z(p\nu,(1-dh)y,(1-p)\Phi(\nu)) \subset \{ (\nu,y) \in A^1
    \times A^1 \}
  \end{equation*}
  is a trivial finite-dimensional vector bundle over $\MC(A/I)$, with
  fibre the image of $hd:A^0\to A^0$, or equivalently, the image of
  $h:A^1\to A^0$.

  Denote the differentials of $x$ and $y:\X\to A^1$ by $\xi$ and
  $\eta\in\Omega_\X\o A^1$.  Taking the differentials of the equations
  defining $\V$ with respect to $x$ and $y$, we obtain the
  differentials
  \begin{align*}
    \om_1 &= (1-p)\xi & \om_2 &= dh\xi - \eta & \om_3 &= d\xi + \ad(\nu+x)\xi .
  \end{align*}
  By the equation
  \begin{equation*}
    (1+h\ad(\nu+x))^{-1}\bigl( \om_1+\om_2+h\om_3 \bigr) = \xi -
    (1+h\ad(\nu+x))^{-1} \eta ,
  \end{equation*}
  we see that the projection from $U\cap\V$ to $\W$ is \'etale,
  proving the lemma.
\end{proof}

We next prove Part 2). Let $b(\bbmu)\in\UA$ be the derivative of
$a(\bbmu)$ with respect to $u$:
\begin{equation*}
  b(\bbmu) =
  \begin{pmatrix}
    \mu_{010} + 2u\mu_{01010} + \dots
    & \mu_{0101} + 2u\mu_{010101} + \dots \\
    1 + \mu_{10} + 2u\mu_{1010} + \dots &
    - \mu_{101} - 2u\mu_{10101} - \dots
  \end{pmatrix}
\end{equation*}
We have the equation
\begin{equation*}
  d_{a(\bbmu)}b(\bbmu) = 1 .
\end{equation*}

Consider the projection $q:\VA\to\VA$ given
by the formula
\begin{equation*}
  q \begin{pmatrix}
    \alpha_{00} & \alpha_{01} \\ \alpha_{10} & \alpha_{11}
  \end{pmatrix} =
  \begin{pmatrix}
    \alpha_{00}(0) & 0 \\ 0 & 0
  \end{pmatrix} ,
\end{equation*}
where $\alpha_{00}(0)$ is the constant term of $\alpha_{00}\in A[u]$.

The homotopy
\begin{align*}
  h &= b(\bbmu)d_{a(\bbmu)}b(\bbmu)(1-q) \\
    &= b(\bbmu)(1-q) - b(\bbmu)^2d_{a(\bbmu)}(1-q)
\end{align*}
maps $\VA^\bull$ to $\VA^\bull$, and satisfies the hypotheses of
Lemma~\ref{smooth}, with respect to the differential
$d_{a(\bbmu)}$: the projection $p$ is given by the explicit formula
\begin{equation*}
  p = 1 - q + b[d_{a(\bbmu)},q] .
\end{equation*}
It follows that the morphism $\MC(C^\bull(\DELTA^1)\o A)\to\MC(A)$ is
smooth at $\bbmu$. This proves Part~2).

Likewise, consider the projection $Q:\VA\to\VA$ given by evaluation at
$u=0$. Applying Lemma~\ref{smooth} to the differential graded algebra
$\VA$, with differential $d_{a(\bbmu)}$, and with homotopy
\begin{align*}
  H &= b(\bbmu)d_{a(\bbmu)}b(\bbmu)(1-Q) \\
    &= b(\bbmu)(1-Q) - b(\bbmu)^2d_{a(\bbmu)}(1-Q) ,
\end{align*}
we see that the morphism
$\MC(C^\bull(\DELTA^1)\o A)\to\MC(C^\bull(\D^1)\o A)$ is smooth at
$\bbmu$. This proves Part~3).
\end{proof}

\end{document}